%% file: arxiv.tex
\documentclass[a4paper,12pt]{article}

\usepackage[margin=1.2in]{geometry}

\usepackage{amsmath}
\usepackage{algorithmic}
\usepackage{algorithm}
\usepackage{amsthm}
\usepackage{amsfonts}
\usepackage{comment}
\usepackage{graphicx}
\usepackage{hyperref}
\usepackage{color}
\usepackage{array}
\usepackage{subcaption}

\include{defs}

\title{Linear Convergence of Frank-Wolfe for Rank-One Matrix Recovery Without Strong Convexity}
\date{}
\author{Dan Garber \\ {\small Technion - Israel Institute of Technology} \\{\small dangar@technion.ac.il}}

\begin{document}

 \maketitle

\begin{abstract}
We consider convex optimization problems which are widely used as convex relaxations for low-rank matrix recovery problems. In particular, in several important problems, such as phase retrieval and robust PCA, the underlying assumption in many cases is that the optimal solution is rank-one. In this paper we consider a simple and natural sufficient condition on the objective so that the optimal solution to these relaxations is indeed unique and rank-one. Mainly, we show that under this condition, the standard Frank-Wolfe method with line-search (i.e., without any tuning of parameters whatsoever), which only requires a single rank-one SVD computation per iteration, finds an $\epsilon$-approximated solution in only $O(\log{1/\epsilon})$ iterations (as opposed to the previous best known bound of $O(1/\epsilon)$), despite the fact that the objective is not strongly convex.
We consider several variants of the basic method with improved complexities, as well as an extension motivated by robust PCA, and  finally, an extension to nonsmooth problems.
\end{abstract}

\section{Introduction}  
Optimization problems in which the goal is to recover a low-rank matrix given certain data/measurements are ubiquitous in machine learning, statistics and related fields. These include for instance the well known \textit{matrix completion} problem \cite{Candes09,Recht11,Jaggi10,ge2016matrix}, the \textit{robust PCA} problem \cite{Candes11,wright2009robust,netrapalli2014non,yi2016fast,mu2016scalable},  matrix formulations of \textit{phase retrieval} problems \cite{candes2015phase,tropp15,Yurtsever17}, and more. While the natural low-rank formulations of these problems are NP-Hard, due to the non-convexity of the rank constraint/objective, all of these problems admit well known and highly popular convex relaxations in which the low-rank constraint is relaxed to a trace-norm constraint which is convex. These convex relaxations are well motivated both empirically and from statistical theory point of view (see above references). On the downside, the scalability of these convex relaxations to high-dimensional instances is questionable, since, despite the implicit assumption that an optimal solution of low-rank should exist, due to the relaxed trace-norm constraint, standard convex optimization methods, such as projected gradient methods \cite{Nesterov13,FISTA} and even conditional gradient-based methods (aka Frank-Wolfe), which are often the ``weapon of choice'' for such problems \cite{Jaggi10,laue2012hybrid,Garber16a,allen2017linear,Yurtsever17,Freund17,Garber18a,Garber18b}, may require in worst-case to compute singular value decompositions (SVD) of high-rank matrices, and/or  to store in memory high-rank matrices, which greatly limits their applicability. Also, since the objective in our case is not strongly convex, exiting analyses of conditional gradient-based methods only give a slow $O(1/\epsilon)$ convergence rate \cite{Garber16a,allen2017linear}.

In this paper, we focus on low-rank matrix optimization problems in which the goal is to recover a rank-one matrix. These include for instance important phase-retrieval problems and several applications of robust PCA, just to name a few. We begin by considering a simple and natural condition that certifies that the convex relaxation indeed admits a unique and rank-one optimal solution. This condition simply requires that at an optimal point, the (minus) gradient matrix admits a non-zero spectral gap between the first and second leading components. Mainly, we show that under this condition, the standard Frank-Wolfe method with line-search converges to an $\epsilon$-approximated solution with number of iterations that scales only with $\log{1/\epsilon}$, as opposed to $1/\epsilon$ in standard Frank-Wolfe analyzes. In particular, we obtain this exponential improvement without requiring the objective to be strongly convex as required in several recent works (e.g., \cite{Garber16a,allen2017linear,Garber18a}). Moreover, our use of the Frank-Wolfe method with line-search does not require any tuning of parameters whatsoever.

Concretely, we consider the following canonical optimization problem:
\begin{eqnarray}\label{eq:optProb}
\min_{\X\in\mS_n}f(\X),
\end{eqnarray} 
where $\mS_n = \{\X\in\mbS^n ~|~\X\succeq 0,~\trace(\X)=1\}$ is the spectrahedron in the space $\mbS^n$ of $n\times n$ real symmetric matrices and we use the standard notation $\X\succeq 0$ to indicate that $\X$ is a positive semidefinite matrix. The function $f:\mbS^n\rightarrow\reals$ is assumed to be convex, and unless stated otherwise it is also assumed to be $\beta$-smooth. 
We let $f^*$ denote the optimal value of Problem \eqref{eq:optProb}.

We refer to Problem \eqref{eq:optProb} as canonical, since it is well known that the highly popular low-rank matrix convex relaxations:
\begin{eqnarray}\label{eq:relatedProbs}
\min_{\Y\in\reals^{m\times n}:~\Vert{\X}\Vert_* \leq \tau}g(\Y) \qquad \textrm{and} \qquad \min_{\Y\in\mbS^n:~\Y\succeq 0,~\trace(\Y)\leq \tau}g(\Y),
\end{eqnarray}
could be directly formulated in the form of Problem \eqref{eq:optProb} (in the above we let $\Vert{\cdot}\Vert_*$ denote the trace-norm, i.e., sum of singular values), see for instance \cite{Jaggi10} \footnote{Here we note that while some problems, such as phase retrieval, are usually formulated as optimization over matrices with complex entries, our results are applicable in a straightforward manner to optimization over the corresponding spectrahedron $\{\X\in\mathbb{C}^{n\times n} ~|~\X\succeq 0,~\trace(\X)=1\}$. However, for simplicity of presentation we focus on matrices with real entries.}. 

We now describe a simple sufficient condition so that the canonical problem \eqref{eq:optProb} indeed admits a unique optimal solution which is also a rank-one matrix. This condition was already suggested in our recent work \cite{Garber19}, however there it was considered for the purpose of controlling the rank of SVD computations required by projected gradient methods to solve problems closely related to \eqref{eq:optProb}, and not for the purpose of obtaining fast convergence rates for globally-convergent methods, which is our main concern in this work.

\begin{assumption}\label{ass:gap}
There exists an optimal solution $\X^*$ to Problem \eqref{eq:optProb} such that $\delta := \lambda_{n-1}(\nabla{}f(\X^*))-\lambda_n(\nabla{}f(\X^*))  > 0$.
\end{assumption}

\begin{lemma}\label{lem:optSolStruct}[Lemma 7 in \cite{Garber19}]
Under Assumption \ref{ass:gap}, Problem \eqref{eq:optProb} admits a unique optimal solution $\X^*$ which is also a rank-one matrix, i.e., $\X^*=\x^*\x^{*\top}$, where $\x^*$ is the eigenvector corresponding to the eigenvalue $\lambda_n(\nabla{}f(\X^*))$.
\end{lemma}

While Assumption \ref{ass:gap} is a sufficient condition for the the existence of a unique and rank-one optimal solution, it is not a necessary condition
. However, the following lemma suggests that this condition is necessary (and also sufficient) for the robustness of the rank of optimal points to arbitrarily-small perturbations.  In particular recall that by the first-order optimality condition it holds that $\Pi_{\mS_n}[\X^*-\beta^{-1}\nabla{}f(\X^*)] = \X^*$, where $\Pi_{\mS_n}[\cdot]$ denotes the Euclidean projection onto $\mS_n$. The lemma is a simple adaptation of Lemma 3 in \cite{Garber19}. A proof is given in the appendix for completeness.

\begin{lemma}\label{lem:gapRobust}
Let $f:\mbS^n\rightarrow\reals$ be $\beta$-smooth and convex. Let $\X^*\in\mS_n$ be an optimal solution of rank-one to the optimization problem $\min_{\X\in\mS_n}f(\X)$. Let $\lambda_1,\dots,\lambda_n$ denote the eigenvalues of $\nabla{}f(\X^*)$ in non-increasing order. Let $\zeta$ be a non-negative scalar. It holds that 
\begin{eqnarray*}
\rank(\Pi_{(1+\zeta)\mS_n}[\X^*-\beta^{-1}\nabla{}f(\X^*)]) > 1 \quad \Longleftrightarrow \quad \zeta >  \beta{}(\lambda_{n-1}-\lambda_n),
\end{eqnarray*}
where $(1+\zeta)\mS_n = \{(1+\zeta)\X~|~\X\in\mS_n\}$, and $\Pi_{(1+\zeta)\mS_n}[\cdot]$ denotes the Euclidean projection onto the convex set $(1+\zeta)\mS_n$.

\end{lemma}
Lemma \ref{lem:gapRobust} shows that an eigen-gap in $\nabla{}f(\X^*)$ implies certain rank-robustness of Problem \eqref{eq:optProb} to small perturbations in the trace bound. In particular, in case the gap in $\nabla{}f(\X^*)$ is zero, we see that an arbitrarily-small perturbation to the trace bound will map an original optimal solution to a higher-rank matrix, which suggests that in such a case, the convex relaxation is ill-posed for the purpose of rank-one matrix recovery.

The following lemma suggests that Assumption \ref{ass:gap} is also robust to certain perturbations in the objective function $f(\cdot)$, which can occur due to noise in the underlying data/misspecification. The proof is given in the appendix.

\begin{lemma}\label{lem:robustGapAss}
Let $f:\mbS^n\rightarrow\reals$ be $\beta$-smooth and convex. Suppose that Assumption \ref{ass:gap} holds w.r.t. $f(\cdot)$ with some parameter $\delta >0$. Let $\tilde{f}:\mbS^n\rightarrow\reals$ be differentiable and convex, and suppose that $\sup_{\X\in\mS_n}\Vert{\nabla{}f(\X) - \nabla\tilde{f}(\X)}\Vert_F \leq \nu$, for some $\nu > 0$. Then, for $\nu <  \frac{1}{2}(1+\frac{2\beta}{\delta})^{-1}\delta$, Assumption \ref{ass:gap} holds w.r.t. the function $\tilde{f}(\cdot)$ with parameter $\tilde{\delta} = \delta -  2\nu(1+\frac{2\beta}{\delta}) > 0$.
\end{lemma}

In Section \ref{sec:exp} we bring empirical motivation for Assumption \ref{ass:gap}.\\

In this paper we leverage Assumption \ref{ass:gap} to derive improved complexities for the Frank-Wolfe method, and certain variants of, all demonstrating linear rate of convergence for Problem \eqref{eq:optProb} (at least under smoothness of $f(\cdot)$). In fact, as we shall show in the sequel (see Lemma \ref{lem:genQG} in Section \ref{sec:fwAlgs}), Assumption \ref{ass:gap} in particular implies that Problem \eqref{eq:optProb} satisfies the quadratic growth property, which is well known to be useful for proving linear convergence rates for first-order methods (see for instance \cite{Drusvyatskiy18,Necoara19}). Nevertheless, even with such a property,  achieving faster rates for Frank-Wolfe-type methods is non-trivial since the standard $O(1/\epsilon)$ rate of the method is not improvable in general, even under strong convexity (see for instance \cite{Garber16a}). Here we should also mention that, while Assumption \ref{ass:gap} implies that the gradient vector of $f(\cdot)$ is non-zero over the feasible set, a property which is known to result in a linear convergence rate for the Frank-Wolfe method whenever the feasible set is strongly convex (see for instance \cite{Demyanov70,Garber2015faster}), in the case of Problem \eqref{eq:optProb} (and also for the related problems appearing in \eqref{eq:relatedProbs}), the feasible set is not strongly convex (or curved), and thus such arguments do not apply in our case.

We focus on the Frank-Wolfe method since i) aside from achieving faster convergence rates, we are also interested in methods that are computationally efficient, and in particular avoid high-rank singular value decompositions (SVD), and ii) the Frank-Wolfe method allows to easily incorporate line-search, which avoids the need to tune parameters, and in particular avoids the need to estimate the eigen-gap in Assumption \ref{ass:gap}. 

Concretely, our main algorithmic result in this paper is the proof of the following theorem, which we currently present only informally.
\begin{theorem}\label{thm:main:informal}[informal]
Under Assumption \ref{ass:gap}, the Frank-Wolfe method with line-search (Algorithm \ref{alg:fw}), finds an $\epsilon$-approximate solution (in function value) to Problem \eqref{eq:optProb}, after $O(\log{1/\epsilon})$ iterations (treating all other quantities, except for the dimension $n$, as constants). Moreover, it also finds in $O(\log{1/\epsilon})$ iterations a unit vector $\v$ such that $\Vert{\v\v^{\top}-\X^*}\Vert_F^2 \leq \epsilon$.
\end{theorem}

A formal and complete description of this result is given in Theorem \ref{thm:rankoneFW} in Section \ref{sec:fwAlgs}. In that section we also present two variants of the Frank-Wolfe method for Problem \eqref{eq:optProb} with improved complexities. In Section \ref{sec:RPCA} we present an extension of Assumption \ref{ass:gap} and Theorem \ref{thm:main:informal} to a class of problems that is motivated by the robust PCA problem and takes the form of minimizing a function of the sum of two blocks of variables, one corresponding to a rank-one matrix, and the other lies in some convex and compact set (see Assumption \ref{ass:gapMix} and Theorem \ref{thm:mainMix}). In Section \ref{sec:nonsmooth} we consider Problem \eqref{eq:optProb} in case the objective function is nonsmooth. Finally, in Section \ref{sec:exp} we present numerical simulations in the support of Assumption \ref{ass:gap} and also a preliminary comparison of the different Frank-Wolfe variants considered in this work.

Table \ref{table:results} gives a quick summary of our results concerning Problem \eqref{eq:optProb}.

After the first version of this work appeared online (see \cite{garber2019linear}), a followup work \cite{ding2020spectral} managed to extend the main result of this paper beyond the rank-one case by providing a Frank-Wolfe-type method with a linear convergence rate under the  assumption that there exists a unique optimal solution $\X^*$ with  $\rank(\X^*)=r^*\geq 1$ and under a natural modification of Assumption \ref{ass:gap}, i.e., that $\X^*$ satisfies an eigen-gap condition of the form: $\lambda_{n-r^*}(\nabla{}f(\X^*)) - \lambda_{n}(\nabla{}f(\X^*)) > 0$ (note that as opposed to the case $r^*=1$ considered here, in case $r^*>1$, this eigen-gap assumption is not sufficient to imply that there exists a unique optimal solution). In \cite{ding2020spectral} it was also established that such an eigen-gap assumption in the gradient at the optimal solution is equivalent to a \textit{strict complementarity} condition for Problem \eqref{eq:optProb}. Additionally, while the linear convergence results in this paper only require convexity and smoothness of $f(\cdot)$ and Assumption \ref{ass:gap}, the result in \cite{ding2020spectral} requires the objective function to be of the form $f(\X) = g(\mA\X)+\langle{\C,\X}\rangle$, where $g(\cdot)$ is smooth and strongly convex, $\mA$ is a linear map, and $\C\in\mbS^n$. Finally, and similarly to \cite{allen2017linear}, the algorithm in \cite{ding2020spectral} requires on each iteration a SVD computation of rank $k\geq{}r^*=\rank(\X^*)$.


\subsection{Additional related work}

In \cite{Zhou2017} the authors have considered an optimization problem closely related to \eqref{eq:optProb}, which takes the form of unconstrained minimization of a smooth convex function plus a nuclear norm regularizer. They showed that under the assumption that the objective is of the form $g(\mA\X)$ where $g$ is smooth and strongly convex and $\mA$ is a linear map, and assuming there exists an optimal solution which satisfies a condition somewhat similar to our Assumption \ref{ass:gap},  a proximal gradient method converges linearly for the problem. While their result allows to consider optimal solutions with arbitrary rank (not only one, as in our case), this current work has three main advantages: i) we do not require that the objective takes the form of strongly convex and smooth function applied to linear map which, while capturing several important applications, is also quite restrictive. Our result only requires the objective to be smooth (and we also obtain a result for nonsmooth problems). ii) \cite{Zhou2017} only establishes a linear convergence rate but does not detail how it depends on the natural parameters of the problem (such as the condition they require on the optimal solution). We on the other-hand, give fully-detailed convergence results with explicit dependency on all relevant parameters. iii) While the linear convergence rate in \cite{Zhou2017} is relevant to proximal gradient methods, these are often not considered the methods of choice for such problems because of the high complexity of computing the proximal step which can require high-rank SVD computations \footnote{in the close proximity of an optimal solution it is quite plausible that only low-rank SVD computations will be needed to compute the proximal step, see for instance our recent work \cite{Garber19}}. On the other hand, here we establish linear convergence rates for the Frank-Wolfe method and simple variants of, which require only rank-one SVD computation on each iteration, and hence are often more suitable for such problems. Moreover, the Frank-Wolfe method can be used with line-search which does not require any parameter tuning.

Finally, it is important to emphasize that there is a very active and recent research effort to analyze nonconvex optimization algorithms for low-rank matrix optimization problems, such as the ones mentioned above, with global convergence guarantees and often with linear convergence rates. However, these results are usually obtained in a statistical setting, in which the data is assumed to follow a very specific and potentially unrealistic statistical  model, see for instance \cite{Prateek10,netrapalli13,jain2014iterative,chen2015fast,bhojanapalli2016global,ge2016matrix} and references therein. On the contrary, in this work, we are free from any statistical assumption/model.

\begin{table*}\renewcommand{\arraystretch}{1.3}
{\footnotesize
\begin{center}
  \begin{tabular}{| p{2.4cm} | p{2.5cm} | p{1.1cm} | p{1.0cm} | p{0.6cm} | p{0.8cm} | p{3.2cm} |} \hline
    Algorithm  & assumption on objective $f(\cdot)$ & requires gap ($\delta$)? & burn-in phase & SVD rank &conv. rate & max iterate rank \\ \hline
FW (Alg \ref{alg:fw}, Thm \ref{thm:rankoneFW}) & smooth & x  & $\frac{\beta^3}{\delta^3}$ &1 & $e^{-\delta{}t/\beta}$ &$\min\{\frac{\beta}{\epsilon},\frac{\beta^3}{\delta^3} + \frac{\beta\log{}1/\epsilon}{\delta}\}$\\ \hline
FW (Alg \ref{alg:fw}, Thm \ref{thm:rankoneFWimprov}) & $g(\mA\X)+\langle{\C,\X}\rangle$, $g$ smooth \& s.c. & x  & $\frac{\beta^3}{\delta^2}$ &1 & $e^{-\delta{}t/\beta}$& $\min\{\frac{\beta}{\epsilon},\frac{\beta^3}{\delta^2} + \frac{\beta\log{}1/\epsilon}{\delta}\}$\\ \hline
FWPG (Alg \ref{alg:fwpg}, Thm \ref{thm:boundedRankAlg}) & smooth & x  & $\frac{\beta^3}{\delta^3}$ & 2& $e^{-\delta{}t/\beta}$ & $\min\{\frac{\beta}{\epsilon},~\frac{\beta^3}{\delta^3}\}$ \\ \hline
FWPG (Alg \ref{alg:fwpg}, Thm \ref{thm:boundedRankAlg}) & $g(\mA\X)+\langle{\C,\X}\rangle$, $g$ smooth \& s.c. & x  & $\frac{\beta^3}{\delta^2}$ &2 & $e^{-\delta{}t/\beta}$& $\min\{\frac{\beta}{\epsilon},~\frac{\beta^3}{\delta^3}\}$ \\ \hline
RegFW (Alg \ref{alg:regfw}, Thm \ref{thm:regFW}) & smooth & \checkmark  & x &1& $e^{-\delta{}t/\beta}$ & $\frac{\beta}{\delta}\log{1/\epsilon}$ \\ \hline
RegFW (Alg \ref{alg:regfw}, Thm \ref{thm:regFW_NS}) & nonsmooth & \checkmark  & x &1 & $e^{-\delta\epsilon{}t}$& $\frac{1}{\delta\epsilon}\log{1/\epsilon}$ \\ \hline
  \end{tabular}  
\caption{Summary of main results. In all cases $f(\cdot)$ is assumed convex. Burn-in phase is number of iterations in which the method converges with standard rate $\beta/t$, before shifting to the fast rate, SVD rank is the rank of SVD used on each iteration, conv. rate is the fast convergence rate after the initial burn-in phase, and max iterate rank gives an upper bound on the number of rank-one components in the representation of the iterate throughout the run, until reaching an $\epsilon$-approximate solution. The result for nonsmooth $f$ (last line), applies to a smooth $\epsilon$-approximation of $f$, see details in Section \ref{sec:nonsmooth}. All results are given in simplified form, omitting all constants except for $\epsilon,\delta,\beta$, and focusing on the most interesting cases. s.c. stands for strongly convex.}\label{table:results}
\end{center}
}
\end{table*}\renewcommand{\arraystretch}{1}

\subsection{Additional notation}
For real matrices we let $\Vert{\cdot}\Vert$ denote the spectral norm (i.e., largest singular value), and we let $\Vert{\cdot}\Vert_F$ denote the Frobenius (Euclidean) norm. For vectors in $\reals^n$ we let $\Vert{\cdot}\Vert_2$ denote the Euclidean norm. In any Euclidean space (e.g., $\reals^n$, $\mbS^n$), we let $\langle{\cdot,\cdot}\rangle$ denote the standard inner product. For a symmetric real matrix $\A\in\mbS^n$, when writing its eigen-decomposition $\A = \sum_{i=1}^n\lambda_i\u_i\u_i^{\top}$, we adopt the standard convention that $\lambda_1 \geq \lambda_2 \geq... \geq\lambda_n$, and that the eigenvectors $\u_1,\dots,\u_n$ form an orthonormal basis for $\reals^n$ (i.e., they have unit norm and are mutually orthogonal).

\subsection{View as a non-linear extension of the leading eigenvector problem}
Our main result (Theorem \ref{thm:main:informal}) could be seen as a faster reduction from nonlinear optimization problems for which the optimal solution is just a leading eigenvector of a certain matrix, to the standard leading eigenvector problem. 

Consider optimization problem \eqref{eq:optProb} in the special case in which $f(\X) = \langle{\X,\A}\rangle$, where $\A\in\mbS^n$, i.e., $f$ is a simple linear function. It is well known that in this case, Problem \eqref{eq:optProb} becomes a tight semidefinite relaxation to computing the leading eigenvector of the matrix $-\A$. In particular, the condition $\lambda_n(\A) < \lambda_{n-1}(\A)$ is sufficient and necessary for this problem to admit a unique optimal solution which is also rank-one (since the leading eigenvector of $-\A$ in this case is unique), i.e., the unique optimal solution is $\X^*=\x^*\x^{*\top}$, where $\x^*$ is the eigenvector corresponding to $\lambda_n(\A)$.

For such $f(\cdot)$ it clearly holds that $\nabla{}f(\X^*)=\A$. Thus, $\X^*$ in particular corresponds to the eigenvector of the smallest eigenvalue of the gradient vector at the optimal solution (or equivalently to the leading eigenvector of $-\nabla{}f(\X^*)$). Moreover, it is well known that standard iterative methods for leading eigenvector computation, such as the well-known power iterations method (see for instance \cite{golub2012matrix}), converge with linear rate when such an eigen-gap exists.

Indeed, Lemma \ref{lem:optSolStruct} shows that for smooth and convex $f$, the condition $\lambda_n(\nabla{}f(\X^*)) < \lambda_{n-1}(\nabla{}f(\X^*))$ is a sufficient condition so that $\X^*$ is a unique optimal solution and also rank-one. In particular, it also corresponds to the eigenvector associated with the smallest eigenvalue $\lambda_n(\nabla{}f(\X^*))$ (or equivalently, the largest eigenvalue of $-\nabla{}f(\X^*)$). We thus refer to such problems as nonlinear eigenvector problems.

Thus, given the arsenal of efficient methods for leading eigenvector computations, it is quite natural to ask if such nonlinear eigenvector problems could be reduced to solving only a short sequence of the standard leading eigenvector problem. The standard Frank-Wolfe analysis (e.g., \cite{Jaggi10}) provides such a reduction, but requires $O(1/\epsilon)$ leading eigenvector computations to find an $\epsilon$-approximated solution (treating all quantities except than $1/\epsilon$ as constants, for simplicity). To the best of our knowledge, Theorem \ref{thm:main:informal} gives the first reduction which requires only $O(\log{1/\epsilon})$ eigenvector computations without requiring the objective function to be strongly convex.

\begin{algorithm}
\caption{Frank-Wolfe with line-search for Problem \eqref{eq:optProb}}
\label{alg:fw}
\begin{algorithmic}[1]
\STATE $\X_1 \gets$ arbitrary point in $\mS_n$
\FOR{$t=1\dots$}
\STATE $\v_t \gets \EV(-\nabla{}f(\X_t))$ \COMMENT{compute an (approximated) leading eigenvector of $-\nabla{}f(\X_t)$}
\STATE choose step size $\eta_t\in[0,1]$ using one of the two options:
\begin{align*}
\textrm{Option 1:}& \quad \eta_t \gets \arg\min_{\eta\in[0,1]}f((1-\eta)\X_t+\eta\v_t\v_t^{\top}) \\
\textrm{Option 2:}& \quad \eta_t \gets \arg\min_{\eta\in[0,1]}f(\X_t) + \eta\langle{\v_t\v_t^{\top}-\X_t,\nabla{}f(\X_t)}\rangle + \frac{\eta^2\beta}{2}\Vert{\X_t-\v_t\v_t^{\top}}\Vert_F^2
\end{align*}
\STATE $\X_{t+1} \gets (1-\eta_t)\X_t+\eta_t\v_t\v_t^{\top}$
\ENDFOR
\end{algorithmic}
\end{algorithm}

\section{Frank-Wolfe-Type Algorithms for Problem \eqref{eq:optProb}}\label{sec:fwAlgs}

\subsection{Proof of Theorem \ref{thm:main:informal}}
We begin with a lemma that will be key to deriving novel bounds on the convergence of Algorithm \ref{alg:fw} under Assumption \ref{ass:gap}. This lemma also establishes that Assumption \ref{ass:gap} implies that Problem \eqref{eq:optProb} satisfies a quadratic growth property.

\begin{lemma}\label{lem:genQG}
Let $\X\in\mS_n$ and suppose that $\lambda_{n-1}(\nabla{}f(\X)) - \lambda_{n}(\nabla{}f(\X)) \geq \delta_{\X}$ for some $\delta_{\X} >0$. Let $\u_n$ be an eigenvector of $\nabla{}f(\X)$ associated with the eigenvalue $\lambda_n(\nabla{}f(\X))$. Then, 
\begin{align}\label{eq:genQG:res1}
\forall \Y\in\mS_n:~ \langle{\Y-\u_n\u_n^{\top},\nabla{}f(\X)}\rangle \geq \delta_{\X}(1-\u_n^{\top}\Y\u_n) \geq \frac{\delta_{\X}}{2}\Vert{\Y-\u_n\u_n^{\top}}\Vert_F^2.
\end{align}
In particular, this implies that if $\X^*$ is an optimal solution for Problem \eqref{eq:optProb} which satisfies Assumption \ref{ass:gap} with parameter $\delta$, then Problem \eqref{eq:optProb} has the quadratic growth property, that is 
\begin{align}\label{eq:genQG:res2}
\forall \Y\in\mS_n:\quad \Vert{\Y-\X^*}\Vert_F^2 \leq \frac{2}{\delta}(f(\Y)-f^*).
\end{align}
\end{lemma}

\begin{proof}
Fix some $\Y\in\mS_n$ and let us write the eigen-decomposition of $\nabla{}f(\X)$ as $\nabla{}f(\X)=\sum_{i=1}^n\lambda_i\u_i\u_i^{\top}$, where the eigenvalues are ordered in non-increasing order. It holds that
\begin{align*}
\langle{\Y-\u_n\u_n^{\top},\nabla{}f(\X)}\rangle &= \langle{\Y-\u_n\u_n^{\top},\sum_{i=1}^n\lambda_i\u_i\u_i^{\top}}\rangle{=} \sum_{i=1}^n\lambda_i\u_i^{\top}\Y\u_i - \lambda_n \nonumber \\
&{\geq} (\lambda_n+\delta_{\X})\sum_{i=1}^{n-1}\u_i^{\top}\Y\u_i + \lambda_n\u_n^{\top}\Y\u_n - \lambda_n \nonumber\\
&{=}\delta_{\X}\sum_{i=1}^{n-1}\u_i^{\top}\Y\u_i {=}\delta_{\X}\left({1 - \u_n^{\top}\Y\u_n}\right),
\end{align*}
where the last two equalities follow since $\sum_{i=1}^n\u_i^{\top}\Y\u_i = 1$.


Now, since since $\Vert{\Y}\Vert_F \leq 1$, we can write
\begin{align}\label{eq:genQG:1}
\langle{\Y-\u_n\u_n^{\top},\nabla{}f(\X)}\rangle &\geq \delta_{\X}(1 - \u_n^{\top}\Y\u_n) \geq \frac{\delta_{\X}}{2}\left({\Vert{\u_n\u_n^{\top}}\Vert^2 + \Vert{\Y}\Vert_F^2 - 2\u_n^{\top}\Y\u_n}\right) \nonumber \\
&= \frac{\delta_{\X}}{2}\Vert{\Y-\u_n\u_n^{\top}}\Vert_F^2.
\end{align}

To prove the quadratic growth consequence under Assumption \ref{ass:gap} (Eq. \eqref{eq:genQG:res2}), we recall that it follows from Lemma \ref{lem:optSolStruct} that $\X^*$ is a rank-one matrix which corresponds to the eigenvector of $\nabla{}f(\X^*)$ associated with the lowest eigenvalue. Thus, by invoking Eq. \eqref{eq:genQG:1} with $\X = \u_n\u_n^{\top} = \X^*$ and $\delta_{\X} = \delta$ (where $\delta$ is as defined in Assumption \ref{ass:gap}), we indeed have that
\begin{align*}
\frac{\delta}{2}\Vert{\Y-\X^*}\Vert_F^2 \leq \langle{\Y-\X^*,\nabla{}f(\X^*)}\rangle \leq f(\Y)-f(\X^*),
\end{align*}
where the last inequality follows from convexity.
\end{proof}

\begin{theorem}[formal version of Theorem \ref{thm:main:informal}]\label{thm:rankoneFW}
Let $\{\X_t\}_{t\geq 1}$ be a sequence produced by Algorithm \ref{alg:fw} and denote for all $t\geq 1$: $h_t:=f(\X_t)-f^*$. Then,
\begin{eqnarray}\label{eq:main:res1}
\forall t\geq 1:\qquad h_t = O\left({\beta/t}\right).
\end{eqnarray}
Moreover, if Assumption \ref{ass:gap} holds then there exists $T_0 = O\left({(\beta/\delta)^3}\right)$ such that
\begin{eqnarray}\label{eq:main:res2}
\forall t\geq T_0: \quad h_{t+1} \leq h_t\left({1-\min\Big\{\frac{\delta}{12\beta},\frac{1}{2}\Big\}}\right).
\end{eqnarray}
Finally, if Assumption \ref{ass:gap} holds then  it also holds that
\begin{eqnarray}\label{eq:main:res3}
\forall t\geq 1: \quad \Vert{\v_t\v_t^{\top}-\X^*}\Vert_F^2 = O\left({\frac{\beta^2}{\delta^3}h_t}\right),
\end{eqnarray}
where $\v_t$ is the eigenvector computed in line 3 of the algorithm.
\end{theorem}

\begin{proof}
The first part of the theorem (Eq. \eqref{eq:main:res1}) follows from standard results on the convergence of Frank-Wolfe with line-search, see for instance \cite{Jaggi13b}.

To prove the second part (Eq. \eqref{eq:main:res2}), we note that using the quadratic growth property (Eq. \eqref{eq:genQG:res2}), we have that for all $t\geq 1$
\begin{align*}
\Vert{\nabla{}f(\X_t)-\nabla{}f(\X^*)}\Vert_F \leq \beta\Vert{\X_t-\X^*}\Vert_F \underset{(a)}{\leq} \beta\sqrt{2\delta^{-1}h_t} \underset{(b)}{=} O\left({\beta\sqrt{\delta^{-1}\beta/t}}\right),
\end{align*}
where (a) follows from Eq. \eqref{eq:genQG:res2} and (b) follows from the first part of the theorem (Eq. \eqref{eq:main:res1}). Thus, for some $T_0 = O((\beta/\delta)^3)$ we have that
\begin{eqnarray}\label{eq:main:0}
\forall t\geq T_0: \quad \Vert{\nabla{}f(\X_t) - \nabla{}f(\X^*)}\Vert_F \leq \frac{\delta}{3}.
\end{eqnarray}

Let us denote the eigen-decomposition of $\nabla{}f(\X_t)$ as $\nabla{}f(\X_t) = \sum_{i=1}^n\lambda_i\u_i\u_i^{\top}$, where the eigenvalues are ordered in non-increasing order. In particular, using Weyl's inequality for the eigenvalues  we have that
\begin{align}\label{eq:main:1}
\lambda_{n-1}-\lambda_n &= \lambda_{n-1}(\nabla{}f(\X^*))-\lambda_n(\nabla{}f(\X^*)) \nonumber \\
&+ (\lambda_{n-1}-\lambda_{n-1}(\nabla{}f(\X^*))) + (\lambda_{n}(\nabla{}f(\X^*)) - \lambda_{n}) \nonumber \\
& \geq  \lambda_{n-1}(\nabla{}f(\X^*))-\lambda_n(\nabla{}f(\X^*)) - 2\Vert{\nabla{}f(\X_t)-\nabla{}f(\X^*)}\Vert_F \nonumber \\
& \geq \delta - \frac{2\delta}{3} = \frac{\delta}{3}.
\end{align}

Let us now recall the Frank-Wolfe update on iteration $t$ of the algorithm: $\X_{t+1} \gets \X_t + \eta_t(\v_t\v_t^{\top}-\X_t)$. Since $\lambda_n < \lambda_{n-1}$, we have that the FW linear subproblem admits a unique optimal solution (the eigenvector $\u_n$), and we can substitute $\v_t$ with $\u_n$, i.e., set $\v_t = \pm{}\u_n$. 

Note that both line-search options in the algorithm imply that
\begin{align}\label{eq:main:LS}
\forall \eta\in[0,1]: ~ f(\X_{t+1}) &\leq f(\X_t) + \eta\langle{\v_t\v_t^{\top}-\X_t),\nabla{}f(\X_t)}\rangle + \frac{\eta^2\beta}{2}\Vert{\v_t\v_t^{\top}-\X_t}\Vert_F^2 \nonumber \\
&= f(\X_t) + \eta\langle{\u_n\u_n^{\top}-\X_t,\nabla{}f(\X_t)}\rangle + \frac{\eta^2\beta}{2}\Vert{\u_n\u_n^{\top}-\X_t}\Vert_F^2,
\end{align}
where the first inequality is due to the smoothness of $f(\cdot)$.

Now, subtracting $f(\X^*)$ from both sides and using Eq. \eqref{eq:genQG:res1} from Lemma \ref{lem:genQG} with $\X = \Y=\X_t$ and  gap $\delta_{\X} = \delta/3$, we have that for all  $0\leq \eta \leq \frac{\delta}{3\beta}$:
\begin{align*}
h_{t+1} &\leq h_t + \left({\eta - \frac{3\eta^2\beta}{\delta}}\right)\langle{\u_n\u_n^{\top}-\X_t,\nabla{}f(\X_t)}\rangle \leq \left({1-  \left({\eta - \frac{3\eta^2\beta}{\delta}}\right)}\right)h_t,
\end{align*}
where the last inequality follows from the convexity of $f(\cdot)$.

We now consider two cases. If $\frac{\delta}{6\beta}\leq 1$, then setting $\eta = \frac{\delta}{6\beta}$ we have that $h_{t+1} \leq \left({1- \frac{\delta}{12\beta}}\right)h_t$.
Otherwise, setting $\eta =1$ and using the fact that $\delta > 6\beta$ we have that $h_{t+1} \leq \frac{1}{2}h_t$.
Overall, we have that for all $t\geq T_0$,
\begin{align*}
h_{t+1} \leq h_t\left({1-\min\{\frac{\delta}{12\beta},\frac{1}{2}\}}\right),
\end{align*}
which proves the second part of the theorem (Eq. \eqref{eq:main:res2}).

Finally, we turn to prove the third part of the theorem  (Eq. \eqref{eq:main:res3}). Recall that under Assumption \ref{ass:gap}, Lemma \ref{lem:optSolStruct} implies that $\X^*$ is a rank-one matrix corresponding to the eigenvector of $\nabla{}f(\X^*)$ associated with the smallest eigenvalue.
By applying Eq. \eqref{eq:genQG:res1} from Lemma \ref{lem:genQG} with $\X=\X^* ,\Y=\v_t\v_t^{\top}$ we have that
\begin{align}\label{eq:main:3}
\Vert{\v_t\v_t^{\top}-\X^*}\Vert_F^2 &\leq \frac{2}{\delta}\langle{\v_t\v_t^{\top}-\X^*,\nabla{}f(\X^*)}\rangle
\nonumber \\
&\underset{(a)}{\leq} \frac{2}{\delta}\langle{\v_t\v_t^{\top}-\X^*,\nabla{}f(\X^*)-\nabla{}f(\X_t)}\rangle \nonumber \\
& \underset{(b)}{\leq} \frac{2\beta}{\delta}\Vert{\v_t\v_t^{\top}-\X^*}\Vert_F\Vert{\X^*-\X_t}\Vert_F \nonumber \\
& \underset{(c)}{\leq} \frac{2\sqrt{2}\beta}{\delta^{3/2}}\Vert{\v_t\v_t^{\top}-\X^*}\Vert_F\sqrt{h_t},
\end{align}
where (a) follows since by definition of $\v_t$ we have that $\langle{\v_t\v_t^{\top}-\X^*,\nabla{}f(\X_t)}\rangle \leq 0$, (b) follows from the Cauchy-Schwarz inequality and the smoothness of $f(\cdot)$, and (c) follows from the quadratic growth property (Eq. \eqref{eq:genQG:res2}).

Rearranging, we indeed get
\begin{align}\label{eq:main:2}
\Vert{\v_t\v_t^{\top}-\X^*}\Vert_F^2 \leq \frac{8\beta^2h_t}{\delta^3}.
\end{align}
\end{proof}

\paragraph{Remark:} Efficient methods for leading eigenvector computations, as required by Algorithm \ref{alg:fw} and all other algorithms we consider in this work, do not produce an accurate solution, but only an approximated leading eigenvector. However, since accounting for these possible approximation errors in the convergence analysis is straight-forward (see for instance \cite{Jaggi10,Jaggi13b,Garber16a}), for ease of presentation we assume all such computations are accurate.

\subsection{Some improvements to Theorem \ref{thm:rankoneFW} under additional structure of objective}

We note that the dependence on $\delta$ in terms of number of iterations until entering the regime of linear convergence (Eq. \eqref{eq:main:res2}) and the distance to the optimal rank-one solution (Eq. \eqref{eq:main:res3}) in Theorem \ref{thm:rankoneFW}, could be quite high (scales with $\delta^{-3}$). We now show that for an important family of structured objective functions, namely those captured by the following Assumption \ref{ass:funcStruct}, this dependence can be improved without changing Algorithm \ref{alg:fw} and with only minor changes to the proof of Theorem \ref{thm:rankoneFW}.

\begin{assumption}\label{ass:funcStruct}
The function $f(\cdot)$ is of the form $f(\X) = g(\mA\X) + \langle{\C,\X}\rangle$, where $\mA:\mbS^n\rightarrow\reals^p$ is a linear map, $g:\reals^p\rightarrow\reals$ is $\beta_g$-smooth and $\alpha_g$-strongly convex, and $\C\in\mbS^n$.
\end{assumption}

We note that Assumption \ref{ass:funcStruct} is also an underlying assumption in \cite{Zhou2017} which, as discussed in the related work section, studied linear convergence rates for proximal gradient methods for the highly related problem of smooth convex minimization with nuclear norm regularization.

In the following we let $\Vert{\mA}\Vert$ denote the operator norm of the map $\mA$, i.e., $\Vert{\mA}\Vert = \max_{\x\in\reals^p,\Vert{\x}\Vert_2=1}\Vert{\mA^{\top}\x}\Vert_F$.

\begin{theorem}\label{thm:rankoneFWimprov}

Let $\{\X_t\}_{t\geq 1}$ be a sequence produced by Algorithm \ref{alg:fw} and denote for all $t\geq 1$: $h_t:=f(\X_t)-f^*$. If both Assumption \ref{ass:gap} and Assumption \ref{ass:funcStruct} hold, then there exists $T_0 =  O\left({\frac{\Vert{\mA}\Vert^3\beta_g^3}{\alpha_g\delta^2}}\right)$ such that
\begin{eqnarray}
\forall t\geq T_0: \quad h_{t+1} \leq h_t\left({1-\min\Big\{\frac{\delta}{12\beta},\frac{1}{2}\Big\}}\right).
\end{eqnarray}
Moreover, 
\begin{eqnarray}
\forall t\geq 1: \quad \Vert{\v_t\v_t^{\top}-\X^*}\Vert_F^2 = O\left({\frac{\beta_g^2\Vert{\mA}\Vert^2}{\alpha_g\delta^2}h_t}\right),
\end{eqnarray}
where $\v_t$ is the eigenvector computed in line 3 of the algorithm.

\end{theorem}

In a nutshell, Theorem \ref{thm:rankoneFWimprov} replaces a factor of $\delta^{-1}$ with a factor of $\alpha_g^{-1}$ in the constant $T_0$ and the RHS of guarantee \eqref{eq:main:res3} in Theorem \ref{thm:rankoneFW}. In order to demonstrate the possible improvement, consider the highly popular case in which the objective function $f(\cdot)$ is a least-squares objective, i.e., $f(\X) = \frac{1}{2}\Vert{\mA\X-\b}\Vert_2^2$, where $\mA$ is a linear map. In this case we have $g(\z) := \frac{1}{2}\Vert{\z-\b}\Vert_2^2$ and so $\alpha_g = 1$, while the eigen-gap $\delta$ can be arbitrarily small. Thus, replacing a factor of $\delta^{-1}$ in favor of $\alpha_g^{-1}$ in the bounds, can be quite significant.

\begin{proof}
Under the additional structural assumption on $f(\cdot)$, it clearly holds that for any $\X,\Y\in\mbS^n$, 
\begin{align*}
\Vert{\nabla{}f(\X) - \nabla{}f(\Y)}\Vert_F = \Vert{\mA^{\top}(\nabla{}g(\mA\X)-\nabla{}g(\mA\Y))}\Vert_F \leq \Vert{\mA}\Vert\beta_g\Vert{\mA\X-\mA\Y}\Vert_2.
\end{align*}
Using the strong convexity of $g(\cdot)$, we have that for all $\X\in\mS_n$, 
\begin{align*}
\Vert{\mA\X-\mA\X^*}\Vert_2 &\leq \sqrt{\frac{2}{\alpha_g}\left({g(\mA\X)+\langle{\C,\X}\rangle-g(\mA\X^*)-\langle{\C,\X^*}\rangle}\right)} \\
&= \sqrt{\frac{2}{\alpha_g}\left({f(\X)-f(\X^*)}\right)}.
\end{align*}
Thus, for any iteration $t$ of Algorithm \ref{alg:fw}, it holds that
\begin{align}\label{eq:improvGradDist:1}
\Vert{\nabla{}f(\X_t) - \nabla{}f(\X^*)}\Vert_F \leq \sqrt{\frac{2\Vert{\mA}\Vert^2\beta_g^2}{\alpha_g}h_t}.
\end{align}

We can now plug-in Eq. \eqref{eq:main:res1} and further obtain that
\begin{align}\label{eq:improvGradDist:2}
\Vert{\nabla{}f(\X_t) - \nabla{}f(\X^*)}\Vert_F = O\left({\sqrt{\frac{\Vert{\mA}\Vert^3\beta_g^3}{\alpha_gt}}}\right),
\end{align}
where we have used the fact that the smoothness parameter of $f$ is at most $\beta_g\Vert{\mA}\Vert$.

Now, we can see that in-order to obtain the bound  \eqref{eq:main:0} in the proof of Theorem \ref{thm:rankoneFW}, it indeed suffices to take $T_0 = O\left({\frac{\Vert{\mA}\Vert^3\beta_g^3}{\alpha_g\delta^2}}\right)$, which proves the first part of the theorem.

To prove the second part, we observe that using \eqref{eq:improvGradDist:1}, Eq. \eqref{eq:main:3} in the proof of Theorem \ref{thm:rankoneFW} could now be replaced with:
\begin{align*}
\Vert{\v_t\v_t^{\top}-\X^*}\Vert_F^2 &\leq  \frac{2}{\delta}\langle{\v_t\v_t^{\top}-\X^*,\nabla{}f(\X^*)-\nabla{}f(\X_t)}\rangle \nonumber \\
& \leq \frac{2}{\delta}\Vert{\v_t\v_t^{\top}-\X^*}\Vert_F\Vert{\nabla{}f(\X^*)-\nabla{}f(\X_t)}\Vert_F \nonumber \\
&{\leq} \frac{2}{\delta}\Vert{\v_t\v_t^{\top}-\X^*}\Vert_F\sqrt{\frac{2\Vert{\mA}\Vert^2\beta_g^2}{\alpha_g}h_t}.
\end{align*}

Rearranging, we get 

\begin{align*}
\Vert{\v_t\v_t^{\top}-\X^*}\Vert_F^2 \leq  \frac{8\beta_g^2\Vert{\mA}\Vert^2}{\alpha_g\delta^2}h_t.
\end{align*}
\end{proof}

\subsubsection{Verifying Assumption \ref{ass:gap} under Assumption \ref{ass:funcStruct}}\label{sec:verify}
We now describe how under Assumption \ref{ass:funcStruct} one can obtain a practically verifiable lower-bound for the parameter $\delta$ in Assumption \ref{ass:gap} (provided that it is indeed greater than zero).

Suppose Assumption \ref{ass:funcStruct} holds and let $\X^*$ denote an optimal solution to Problem \eqref{eq:optProb}. Combining Weyl's inequality for the eigenvalues and Eq. \eqref{eq:improvGradDist:1} from the proof of Theorem \ref{thm:rankoneFWimprov}, we have that for any $t\geq 1$, the matrix $\X_t$ from Algorithm \ref{alg:fw}, satisfies
\begin{align*}
\forall i\in[n]: \quad \vert{\lambda_i(\nabla{}f(\X_t)) - \lambda_i(\nabla{}f(\X^*))}\vert \leq \sqrt{\frac{2\Vert{\mA}\Vert^2\beta_g^2}{\alpha_g}(f(\X_t)-f(\X^*))}.
\end{align*}
This implies that
\begin{align}\label{eq:eigengapEst}
&\vert{(\lambda_{n-1}(\nabla{}f(\X^*)) - \lambda_n(\nabla{}f(\X^*))) - (\lambda_{n-1}(\nabla{}f(\X_t)) - \lambda_n(\nabla{}f(\X_t)))}\vert \leq \nonumber \\
&\leq 2 \sqrt{\frac{2\Vert{\mA}\Vert^2\beta_g^2}{\alpha_g}(f(\X_t)-f(\X^*))}.
\end{align}

Suppose now that Assumption \ref{ass:gap} indeed holds with parameter $\delta > 0$. Using Eq. \eqref{eq:eigengapEst}, once we arrive at an iteration $t$ for which it holds that the RHS of \eqref{eq:eigengapEst} is smaller for instance than $\delta/3$, by computing the eigenvalues $\lambda_{n-1}(\nabla{}f(\X_t)), \lambda_n(\nabla{}f(\X_t))$\footnote{extending this discussion to the case in which these eigenvalues are only approximated up to sufficient precision is straightforward}, we can verify that Assumption \ref{ass:gap} holds with parameter at least $\delta/3$, and in particular that there exists a unique optimal solution which is also rank-one.

Note that in order to verify that the RHS of \eqref{eq:eigengapEst} is indeed smaller than $\delta/3$, it suffices to replace the approximation error $f(\X_t)-f(\X^*)$ with the simple upper-bound $\langle{\X_t-\v_t\v_t^{\top},\nabla{}f(\X_t)}\rangle$, where $\v_t$ is the eigenvector computed on iteration $t$ of Algorithm \ref{alg:fw}\footnote{this quantity is known as the duality gap and it is indeed an upper-bound on the approximation error since $f(\cdot)$ is convex, see \cite{Jaggi13b}}. Note that for Algorithm \ref{alg:fw}, it is known that the non-negative quantity $\langle{\X_t-\v_t\v_t^{\top},\nabla{}f(\X_t)}\rangle$ also converges to zero as a function of $t$, with rate at least $O(1/t)$ (without requiring Assumption \ref{ass:gap}) \cite{Jaggi13b}.


\subsection{Bounded-rank algorithm}
Despite the linear convergence result for the Frank-Wolfe method detailed in Theorem \ref{thm:rankoneFW}, still a certain disadvantage is that the rank of the iterates (or number of rank-one components that needs to be stored in memory to maintain a factorization of the current iterate $\X_t$) grows linearly with the iteration counter $t$. We now suggest a simple modification, that actually combines the Frank-Wolfe method and the projected gradient method, and guarantees that the number of rank-one components is always bounded and is independent of $1/\epsilon$, where $\epsilon$ is the target accuracy. 

The main idea is to use the recent results in \cite{Garber19} which show that under Assumption \ref{ass:gap}, in a ball of radius $\Theta(\delta/\beta)$ around $\X^*$, the projected gradient method, when applied to Problem \eqref{eq:optProb}, will always produce iterates that are rank-one. Moreover, whether the projection is indeed rank-one or not could be verified by examining the first and second leading eigenvalues of the corresponding matrix. This leads to an algorithm that  applies either conditional gradient steps or projected gradient steps (when the projection is rank-one), until entering the above mentioned ball around $\X^*$. Once the iterates are inside the ball, it is guaranteed that only projected gradient steps which result in a rank-one matrix will be used, and thus from this point on, only a single rank-one matrix needs to be stored in memory.

This modification comes with the price that now each iteration of the algorithm (see Algorithm \ref{alg:fwpg} below) requires, in worst case, a rank-two SVD computation of a $n\times n$ matrix, and an additional one leading eigenvector computation.

\begin{algorithm}
\caption{Frank-Wolfe meets Projected Gradient for Problem \eqref{eq:optProb}}
\label{alg:fwpg}
\begin{algorithmic}[1]
\STATE input: smoothness parameter $\beta$
\STATE let $\X_1$ be an arbitrary point in $\mS_n$
\FOR{$t=1\dots$}
\STATE $\Y_{t+1} \gets \X_t - \frac{1}{\beta}\nabla{}f(\X_t)$
\STATE  let $\lambda_1\u_1\u_1^{\top}+\lambda_2\u_2\u_2^{\top}$ be the rank-two truncated eigen-decomposition of $\Y_{t+1}$ (i.e., taking the two leading components with largest eigenvalues)
\IF{$\lambda_1 \geq 1 + \lambda_2$}
\STATE $\X_{t+1} \gets \u_1\u_1^{\top}$
\ELSE
\STATE $\v_t \gets \EV(-\nabla{}f(\X_t))$
\STATE choose step size $\eta_t\in[0,1]$ using one of the two options:
\begin{align*}
\textrm{Option 1:}& \quad \eta_t \gets \arg\min_{\eta\in[0,1]}f((1-\eta)\X_t+\eta\v_t\v_t^{\top}) \\
\textrm{Option 2:}& \quad \eta_t \gets \arg\min_{\eta\in[0,1]}f(\X_t) + \eta\langle{\v_t\v_t^{\top}-\X_t,\nabla{}f(\X_t)}\rangle + \frac{\eta^2\beta}{2}\Vert{\X_t-\v_t\v_t^{\top}}\Vert_F^2
\end{align*}
\STATE $\X_{t+1} \gets (1-\eta_t)\X_t+\eta_t\v_t\v_t^{\top}$
\ENDIF
\ENDFOR
\end{algorithmic}
\end{algorithm}

\begin{theorem}\label{thm:boundedRankAlg}
The sequence $\{\X_t\}_{t\geq 1}$ produced by Algorithm \ref{alg:fwpg} has all the guarantees stated in Theorem \ref{thm:rankoneFW} (or Theorem \ref{thm:rankoneFWimprov} if Assumption \ref{ass:funcStruct} also holds). Moreover, there exists $T_1 = O\left({(\beta/\delta)^3}\right)$, such that for all $t\geq T_1$ it holds that $\rank(\X_t) =1$. 

\end{theorem}
\begin{proof}
Note that according to the structure of the Euclidean projection over $\mS_n$ (see for instance Lemma 6 in \cite{Garber19}), when the condition in the if statement (line 6 of the algorithm) holds on some iteration $t$, then indeed the projection of $\Y_{t+1}$ onto $\mS_n$ is given by the rank-one matrix $\u_1\u_1^{\top}$, and thus in this case $\X_{t+1}$ is equivalent to the standard projected gradient update step: $\X_{t+1} \gets \Pi_{\mS_n}[\X_t-\beta^{-1}\nabla{}f(\X_t)]$. Thus, Algorithm \ref{alg:fwpg} either applies a standard projected gradient update (when the projection is rank-one), or otherwise a Frank-Wolfe update with line-search.

In particular, if $\X_{t+1} \gets \Pi_{\mS_n}[\X_t-\beta^{-1}\nabla{}f(\X_t)]$ then, as it is well known, we have that for any $\Y\in\mS_n$,
\begin{align*}
f(\X_{t+1})  &\leq f(\X_t) + \langle{\X_{t+1}-\X_t,\nabla{}f(\X_t)}\rangle + \frac{\beta}{2}\Vert{\X_{t+1}-\X_t}\Vert_F^2 \\
&= f(\X_t)+ \frac{\beta}{2}\Vert{\X_{t+1}-(\X_t-\beta^{-1}\nabla{}f(\X_t))}\Vert_F^2 - \frac{1}{2\beta}\Vert{\nabla{}f(\X_t)}\Vert_F^2 \\
&\leq f(\X_t) + \frac{\beta}{2}\Vert{\Y-(\X_t-\beta^{-1}\nabla{}f(\X_t))}\Vert_F^2 - \frac{1}{2\beta}\Vert{\nabla{}f(\X_t)}\Vert_F^2 \\
&= f(\X_t)  + \langle{\Y-\X_t,\nabla{}f(\X_t)}\rangle + \frac{\beta}{2}\Vert{\Y-\X_t}\Vert_F^2.
\end{align*}

In particular, for any $\eta_t\in[0,1]$, setting $\Y = (1-\eta_t)\X_t + \eta_t\v_t\v_t^{\top}$, with $\v_t$ being the leading eigenvector of $-\nabla{}f_t(\X_t)$, we obtain
\begin{align*}
f(\X_{t+1}) &\leq f(\X_t)  + \eta_t\langle{\v_t\v_t^{\top}-\X_t,\nabla{}f(\X_t)}\rangle + \frac{\eta_t^2\beta}{2}\Vert{\X_t-\v_t\v_t^{\top}}\Vert_F^2
\end{align*}
(which is the same as Eq. \eqref{eq:main:LS} in the proof of Theorem \ref{thm:rankoneFW}).

Thus, a projected gradient update enjoys a per-iteration worst-case error reduction that is no worse than that of a Frank-Wolfe step with line-search (option 2 in Algorithm \ref{alg:fw}). This implies that both the $O(1/t)$ convergence rate and the linear convergence rates for the sequence $\{\X_t\}_{t\geq 1}$ guaranteed in Theorems \ref{thm:rankoneFW} and \ref{thm:rankoneFWimprov}, also hold for Algorithm \ref{alg:fwpg}.

In particular, using the quadratic growth property (Eq. \eqref{eq:genQG:res2}) together with the $O(\beta/t)$ convergence rate (Eq. \eqref{eq:main:res1}), we have that there exists $T_1 = O((\beta/\delta)^3)$ such that for all $t\geq T_1$, 
\begin{align}\label{eq:boundedRank:1}
\Vert{\X_t - \X^*}\Vert_F^2  = O\left({\frac{\beta}{\delta{}T_1}}\right) \leq  \left({\frac{\delta}{4\beta}}\right)^2.
\end{align}

Thus, starting from iteration $T_1$ and onwards, all iterates of the algorithm lie inside the Euclidean ball of radius $\delta/(4\beta)$ around $\X^*$. Invoking Theorem 7 in \cite{Garber19}, it is guaranteed that once the iterates are inside this ball, the projection of $\Y_{t+1}$ onto $\mS_n$ is indeed rank-one, or equivalently, the condition on the eigenvalues in line 6 of the algorithm, always holds.
\end{proof}

\subsection{No burn-in phase when gap is known}
Another disadvantage of Theorems \ref{thm:rankoneFW}, \ref{thm:rankoneFWimprov} is that the linear convergence applies only after a certain ``burn-in" phase. Here we show that if an estimate for the eigen-gap $\delta = \lambda_{n-1}(\nabla{}f(\X^*)) - \lambda_n(\nabla{}f(\X^*))$ is available, then it is possible to modify the Frank-Wolfe method, without essentially changing the complexity of each iteration, so that it enjoys a global linear convergence rate. This modification and convergence analysis follows in an almost straight-forward manner from the work \cite{allen2017linear}, when combined with the quadratic growth property (Lemma \ref{lem:genQG}).

\begin{algorithm}
\caption{Regularized Frank-Wolfe for Problem \eqref{eq:optProb}}
\label{alg:regfw}
\begin{algorithmic}[1]
\STATE input: smoothness parameter $\beta$, gap estimate $\hat{\delta} \in (0, \lambda_{n-1}(\nabla{}f(\X^*)) - \lambda_n(\nabla{}f(\X^*))]$ 
\STATE let $\X_1$ be an arbitrary point in $\mS_n$
\STATE $\eta \gets \min\{1, \frac{\hat{\delta}}{2\beta}\}$
\FOR{$t=1\dots$}
\STATE $\v_t \gets \arg\min_{\Vert{\v}\Vert=1}\langle{\v\v^{\top},\nabla{}f(\X_t)}\rangle + \frac{\eta\beta}{2}\Vert{\v\v^{\top}-\X_t}\Vert_F^2$ \COMMENT{note this is equivalent to $\v_t \gets \EV\left({-\nabla{}f(\X_t)+\eta\beta\X_t}\right)$}
\STATE $\X_{t+1} \gets (1-\eta)\X_t+\eta\v_t\v_t^{\top}$
\ENDFOR
\end{algorithmic}
\end{algorithm}

\begin{theorem}\label{thm:regFW}
Under Assumption \ref{ass:gap}, the iterates of Algorithm \ref{alg:regfw} satisfy
\begin{align*}
\forall t\geq 1: \quad f(\X_{t+1}) -f^*  \leq \left({1 - \min\{\frac{1}{2},\frac{\hat{\delta}}{4\beta}\}}\right)\left({f(\X_t) - f^*}\right).
\end{align*}
\end{theorem}
As discussed, the proof is a simple application of the arguments used in \cite{allen2017linear} and Lemma \ref{lem:genQG}, however since it is very short, we include it here for completeness.
\begin{proof}
On any iteration $t$ it holds that
\begin{align}\label{eq:thm:regFW:0}
f(\X_{t+1}) - f^* &\underset{(a)}{\leq} f(\X_t) -f^* + \eta\langle{\v_t\v_t^{\top}-\X_t,\nabla{}f(\X_t)}\rangle + \frac{\eta^2\beta}{2}\Vert{\v_t\v_t^{\top}-\X_t}\Vert_F^2 \nonumber \\
&\underset{(b)}{\leq} f(\X_t) -f^* + \eta\langle{\X^*-\X_t,\nabla{}f(\X_t)}\rangle + \frac{\eta^2\beta}{2}\Vert{\X^*-\X_t}\Vert_F^2 \nonumber\\
&\underset{(c)}{\leq} \left({f(\X_t) -f^*}\right)\left({1 - \eta + \frac{\eta^2\beta}{\delta}}\right) \nonumber \\
&\underset{(d)}{\leq} \left({f(\X_t) -f^*}\right)\left({1 -\eta + \frac{\eta^2\beta}{\hat{\delta}}}\right),
\end{align}
where (a) follows from smoothness of $f$, (b) follows from the optimal choice of $\v_t$ and since, under Assumption \ref{ass:gap}, $\X^*$ is rank-one, (c) follows from convexity of $f(\cdot)$ and  Eq. \eqref{eq:genQG:res2} in Lemma \ref{lem:genQG}, and (d) follows since $\hat{\delta} \leq \delta$.

Now, if $\frac{\hat{\delta}}{2\beta}\leq 1$, then plugging-in $\eta = \frac{\hat{\delta}}{2\beta}$ into the RHS of \eqref{eq:thm:regFW:0}, we have that
\begin{align*}
f(\X_{t+1})-f^* \leq \left({f(\X_t) - f^*}\right)\left({1-\frac{\hat{\delta}}{4\beta}}\right).
\end{align*}
Otherwise, we have that $\eta =1$ and $\frac{\beta}{\hat{\delta}} < \frac{1}{2}$. In this case, plugging-in $\eta=1$ into the RHS of \eqref{eq:thm:regFW:0}, we have that
\begin{align*}
f(\X_{t+1})-f^* \leq \frac{1}{2}\left({f(\X_t) - f^*}\right).
\end{align*}
Combining these two cases yields the theorem.
\end{proof}

\paragraph{Remark:} It is possible to combine the use of the projected gradient method, as applied in Algorithm \ref{alg:fwpg}, and the regularized Frank-Wolfe update, as applied in Algorithm \ref{alg:regfw}, to obtain an algorithm that has both bounded rank and global linear convergence rate. This derivation is quite straightforward given these two ingredients and we omit it.

\section{Extension Motivated by Robust-PCA}\label{sec:RPCA}
We now consider the following extension of Problem \eqref{eq:optProb}.

\begin{eqnarray}\label{eq:optProbMix}
\min_{\X\in\mS_n,\y\in\mK}\{f(\X,\y):=g(\mA\X+\y) + \langle{\C,\X}\rangle + \langle{\c,\y}\rangle\},
\end{eqnarray}
where $g:\reals^p\rightarrow\reals$ is assumed $\alpha_g$-strongly convex and $\beta_g$-smooth, $\mA:\mbS^n\rightarrow\reals^p$ is a linear map, $\mK\subset\reals^p$ is assumed convex and compact, and $\C\in\mbS^n,\c\in\reals^p$. Throughout this section we use $D_{\mK}$ to denote the Euclidean diameter of $\mK$.

For instance, the Robust-PCA problem \cite{Candes11,mu2016scalable,Garber18a}:
\begin{eqnarray*}
\min_{\X\in\reals^{m\times n}:\Vert{\X}\Vert_* \leq \tau,~\Y\in\reals^{m\times n}:\Vert{\Y}\Vert_1 \leq k}\frac{1}{2}\Vert{\X+\Y-\M}\Vert_F^2,
\end{eqnarray*}
where $\M\in\reals^{m\times n}$ is some input matrix, and $\Vert{\cdot}\Vert_1$ is the standard entry-wise $\ell_1$ norm, could be formulated as Problem \eqref{eq:optProbMix} 
 via standard transformations (see for instance \cite{Jaggi10}).

Another relevant example is that of phase retrieval with corrupted measurements, in which case the vector $\y$ accounts for the corruptions, and $\mK$ can be taken to be some norm-induced ball (e.g., $\ell_1$ ball in case of sparse corruptions).

In the sequel, we let $\nabla_{\X}f(\X,\Y)$ denote the derivative of $f$ w.r.t. the block $\X$ and $\nabla{}f_{\y}f(\X,\Y)$ the derivative w.r.t. $\y$. Also, as before, we denote $\Vert{\mA}\Vert = \max_{\x\in\reals^p,\Vert{\x}\Vert_2=1}\Vert{\mA^{\top}\x}\Vert_F$.

Towards extending our results for Problem \eqref{eq:optProb} to Problem \eqref{eq:optProbMix}, we consider a standard first-order method which combines the use of Frank-Wolfe with line-search in order to update the matrix variable $\X$ (as done for Problem \eqref{eq:optProb}) with the standard projected gradient method for updating the variable $\y$ \footnote{Here we make an implicit assumption that it is computationally efficient to compute Euclidean projections onto the set $\mK$.}. See Algorithm \ref{alg:mixed}. 
\begin{algorithm}[H]
\caption{Projected Gradient combined with Frank-Wolfe for Problem \eqref{eq:optProbMix}}
\label{alg:mixed}
\begin{algorithmic}[1]
\STATE input: smoothness parameter $\beta_g$
\STATE $(\X_1,\y_1)\gets$ arbitrary point in $\mS_n\times\mK$
\FOR{$t=1\dots$}
\STATE $\y_{t+1} \gets \Pi_{\mK}[\y_t-\frac{1}{2\beta_g}\nabla_{\y}f(\X_t,\y_t)]$
\STATE $\v_t \gets \EV(-\nabla_{\X}f(\X_t,\y_t))$
\STATE $\eta_t \gets \arg\min_{\eta\in[0,1]}f((1-\eta)\X_t+\eta\v_t\v_t^{\top},\y_{t+1})$
\STATE $\X_{t+1} \gets (1-\eta_t)\X_t+\eta_t\v_t\v_t^{\top}$
\ENDFOR
\end{algorithmic}
\end{algorithm}

Working towards proving an analogue of Theorem \ref{thm:rankoneFW} for Problem \eqref{eq:optProbMix}, we begin by extending  our underlying gap assumption to the new setting.

\begin{lemma}\label{lem:constGrad}
Let $\mW^*\subset\mS_n\times\mK$ denote the set of optimal solutions to Problem \eqref{eq:optProbMix}. Then $\nabla{}f(\X,\y)$ is constant over $\mW^*$.
\end{lemma}
\begin{proof}
Since $g$ is strongly convex it follows that $\mA\X+\y$ is constant over $\mW^*$. Note that for any $\X,\y$, $\nabla_{\X}f(\X,\y) = \mA^{\top}\nabla{}g(\mA\X+\y)+\C$, $\nabla_{\y}f(\X,\y) = \nabla{}g(\mA\X+\y)+\c$. Hence, it follows that indeed $\nabla{}f$ is constant over $\mW^*$.
\end{proof}


\begin{assumption}\label{ass:gapMix}
The gradient vector at every optimal solution $(\X^*,\y^*)$ satisfies: $\lambda_{n-1}(\nabla_{\X}f(\X^*,\y^*))-\lambda_n(\nabla_{\X}f(\X^*,\y^*)) = \delta > 0$.\footnote{Recall that according to the previous lemma the gradient vector is constant over the set of optimal solutions and thus, this is equivalent to assuming the eigen-gap holds for some optimal solution.}
\end{assumption}

\begin{lemma}\label{lem:optSolStructMix}
Under Assumption \ref{ass:gapMix} there exists a unique optimal solution $(\X^*,\y^*)$ to Problem \eqref{eq:optProbMix}. Moreover, $\X^*$ is rank-one, that is $\X^*=\x^*\x^{*\top}$ for some unit vector $\x^*\in\reals^n$.
\end{lemma}

\begin{proof}
Fix some optimal solution $(\X^*,\y^*)$ and consider the function $q(\X) = f(\X,\y^*)$. Clearly $\nabla{}q(\X) = \nabla_{\X}f(\X,\y^*)$ and $\X^*\in\arg\min_{\X\in\mS_n}q(\X)$. Thus, according to Assumption \ref{ass:gapMix} it follows that $\lambda_{n-1}(\nabla{}q(\X^*)) - \lambda_n(\nabla{}q(\X^*)) = \delta >0$. Thus, by Lemma \ref{lem:optSolStruct} it follows that $\X^*$ is the unique minimizer of $q(\X)$ over $\mS_n$, and moreover, $\X^*=\x^*\x^{*\top}$ is rank-one, where $\x^*$ is the eigenvector which corresponds to the eigenvalue $\lambda_n(\nabla{}q(\X^*))=\lambda_n(\nabla_{\X}f(\X^*))$. However, by Lemma \ref{lem:constGrad}, the gradient vector of $f(\cdot,\cdot)$ is constant over the optimal set, and hence, if there exists another optimal solution $(\X_2^*,\y_2^*)$ to Problem \eqref{eq:optProbMix}, by the above reasoning it must hold that $\X^*_2 = \X^*=\x^*\x^{*\top}$.

Now, since $g(\cdot)$ is strongly convex it follows that the vector $\mA\X+\y$ is constant over the optimal set $\mW^*\subseteq\mS_n\times\mK$. Thus, for any two optimal solutions $(\X^*,\y_1^*),(\X^*,\y_2^*)$ we have that
\begin{eqnarray*}
\y_1^*-\y_2^* = (\mA\X^*+\y_1^*) - (\mA\X^*+\y_2^*) = \mathbf{0}.
\end{eqnarray*}
Hence, the lemma follows.
\end{proof}

The proof of the following lemma follows essentially from the same arguments used to derive Eq. \eqref{eq:improvGradDist:1} and thus we omit it.
\begin{lemma}\label{lem:gradDistMix}
For any $(\X,\y)\in\mS_n\times\mK$ and optimal solution $(\X^*,\y^*)\in\mS_n\times\mK$ it holds that
\begin{eqnarray}\label{eq:gradDist}
\Vert{\nabla_{\X}f(\X,\y) - \nabla_{\X}f(\X^*,\y^*)}\Vert_F \leq  \frac{\sqrt{2}\beta_g\Vert{\mA}\Vert}{\sqrt{\alpha_g}}\sqrt{f(\X,\y)-f^*}.
\end{eqnarray}
\end{lemma}




We can now finally present and prove our main result for Problem \eqref{eq:optProbMix}.

\begin{theorem}\label{thm:mainMix}
Let $\{(\X_t,\y_t)\}_{t\geq 1}$ be a sequence produced by Algorithm \ref{alg:mixed} and denote for all $t\geq 1$: $h_t:=f(\X_t,\y_t)-f^*$. Then,
\begin{align}\label{eq:mainMix:res1s}
\forall t\geq 1:\qquad h_t = O\left({\frac{\beta_g(\Vert{\mA}\Vert^2+D_{\mK}^2)}{t}}\right).
\end{align}
Moreover, under Assumption \ref{ass:gapMix}, there exists $T_0 = O\left({\frac{\Vert{\mA}\Vert^2\beta_g^3(\Vert{\mA}\Vert^2+D_{\mK}^2)}{\alpha_g\delta^2}}\right)$ such that $\forall t\geq T_0$:
\begin{align}\label{eq:mainMix:res2s}
h_{t+1} \leq h_t\left({1 - \min\{\frac{1}{6},\frac{\alpha_g\delta^2}{4\beta_g\left({10\alpha_g\delta\Vert{\mA}\Vert^2+4\delta^2+64\Vert{\mA}\Vert^4\beta_g^2}\right)},\frac{\delta}{72\beta_g\Vert{\mA}\Vert^2}\}}\right).
\end{align}

Finally, under Assumption \ref{ass:gapMix}, it also holds that
\begin{align}\label{eq:mainMix:res3s}
\forall t\geq 1: \quad \Vert{\v_t\v_t^{\top}-\X^*}\Vert_F^2 = O\left({\frac{\beta_g^2\Vert{\mA}\Vert^2}{\alpha_g\delta^2}h_t}\right),
\end{align}
where $\v_t$ is the eigenvector computed in line 5 of the algorithm.
\end{theorem}

\begin{proof}
Fix some iteration $t$. By the optimal choice of $\eta_t$, we have that for any $\eta_{\X}\in[0,1]$ it holds that
\begin{eqnarray*}
f(\X_{t+1},\y_{t+1}) \leq f((1-\eta_{\X})\X_t+\eta_{\X}\v_t\v_t^{\top},\y_{t+1}).
\end{eqnarray*}

We introduce the notation $\z_t  = \mA\X_t+\y_t$. Using the smoothness of $g(\cdot)$, it holds for any $\eta_{\X}\in[0,1]$ that
\begin{align*}
f(\X_{t+1},\y_{t+1}) & \leq g(\mA\X_t+\mA(\eta_{\X}(\v_t\v_t^{\top}-\X_t))+\y_{t} + (\y_{t+1}-\y_t)) \\
&+\langle{\C, \X_t + \eta_{\X}(\v_t\v_t^{\top}-\X)}\rangle + \langle{\c,\y_{t+1}}\rangle\\
&\leq g(\z_t) + \left({\mA(\eta_{\X}(\v_t\v_t^{\top}-\X_t)) + \y_{t+1}-\y_t}\right)^{\top}\nabla{}g(\z_t) \\
&+ \frac{\beta_g}{2}\Vert{\mA(\eta_{\X}(\v_t\v_t^{\top}-\X_t)) + \y_{t+1}-\y_t}\Vert_2^2 \\
&+\langle{\C, \X_t + \eta_{\X}(\v_t\v_t^{\top}-\X)}\rangle + \langle{\c,\y_{t+1}}\rangle\\
&\leq f(\X_t,\y_t) + \eta_{\X}\langle{\v_t\v_t^{\top}-\X_t,\nabla{}f_{\X}(\X_t,\y_t)}\rangle \\
&+ (\y_{t+1}-\y_t)^{\top}\nabla_{\y}f(\X_t,\y_t)\\
& + \beta_g\left({\Vert{\mA}\Vert^2\eta_{\X}^2\Vert{\v_t\v_t^{\top}-\X_t}\Vert_F^2 + \Vert{\y_{t+1}-\y_t}\Vert_2^2}\right),
\end{align*}
where in the last inequality we have used the triangle inequality for the Euclidean norm and $(a+b)^2 \leq 2a^2+2b^2$.

From the choice of $\y_{t+1}$, it follows that for all $\eta_{\y}\in[0,1]$,
\begin{align*}
&(\y_{t+1}-\y_t)^{\top}\nabla_{\y}f(\X_t,\y_t) + \beta_g \Vert{\y_{t+1}-\y_t}\Vert_2^2 \leq \\ 
&((\y_t + \eta_{\y}(\y^*-\y_t))-\y_t)^{\top}\nabla_{\y}f(\X_t,\y_t) + \beta_g \Vert{(\y_t + \eta_{\y}(\y^*-\y_t))-\y_t}\Vert_2^2 =\\
&\eta_{\y}(\y^*-\y_t)^{\top}\nabla_{\y}f(\X_t,\y_t) + \beta_g\eta_{\y}^2 \Vert{\y^*-\y_t}\Vert_2^2.
\end{align*}

Combining both inequalities we have that for any $(\eta_{\X},\eta_{\y})\in[0,1]\times[0,1]$,
\begin{align}\label{eq:mainMix:0}
f(\X_{t+1},\y_{t+1}) &\leq f(\X_t,\y_t) + \eta_{\X}\langle{\v_t\v_t^{\top}-\X_t,\nabla_{\X}f(\X_t,\y_{t})}\rangle \nonumber \\
&+ \eta_{\y}(\y^*-\y_t)^{\top}\nabla_{\y}f(\X_t,\y_t)  \nonumber \\
&+\beta_g\left({\eta_{\X}^2\Vert{\mA}\Vert^2\Vert{\v_t\v_t^{\top}-\X_t}\Vert_F^2 + \eta_{\y}^2\Vert{\y^*-\y_t}\Vert_2^2}\right).
\end{align}

Now, part one of the Theorem (Eq. \eqref{eq:mainMix:res1s}) follows from setting the standard observation that $\langle{\v_t\v_t^{\top},\nabla_{\X}f(\X_t,\y_t)}\rangle \leq \langle{\X^{*},\nabla_{\X}f(\X_t,\y_t)}\rangle$, and from here the $O(1/t)$ rate follows from standard arguments involving the convexity of $f(\cdot)$ and the fact that $\mS_n,\mK$ are bounded.

We now continue to prove the second part of the theorem (Eq. \eqref{eq:mainMix:res2s}). Note that from Lemma \ref{lem:gradDistMix} and the first part of the theorem, it follows that there exists $T_0 = O\left({\frac{\Vert{\mA}\Vert^2\beta_g^3(\Vert{\mA}\Vert^2+D_{\mK}^2)}{\alpha_g\delta^2}}\right)$ such that
\begin{eqnarray}\label{eq:mainMix}
\forall t\geq T_0: \qquad \Vert{\nabla_{\X}f(\X_t,\y_t)-\nabla_{\X}f(\X^*,\y^*)}\Vert_F \leq \frac{\delta}{3}.
\end{eqnarray}

Throughout the rest of the proof we focus on some iteration $t\geq T_0$. Denote $\z^* = \mA\X^*+\y^*$.

Observe that
\begin{eqnarray}\label{eq:mainMix:1}
\Vert{\y^*-\y_t}\Vert_2 &=&\Vert{(\z^*-\mA\X^*) - (\z_t-\mA\X_t)}\Vert_2 \leq \Vert{\z^*-\z_t}\Vert_2+ \Vert{\mA\X^*-\mA\X_t}\Vert_2 \nonumber  \\
&\leq& \sqrt{\frac{2}{\alpha_g}}\sqrt{f(\X_t,\y_t)-f^*} + \Vert{\mA}\Vert\Vert{\X^*-\X_t}\Vert_F,
\end{eqnarray}
where the last inequality follows from the strong convexity of $g(\cdot)$.

Let us write the eigen-decomposition of $\nabla{}_{\X}f(\X_t,\y_t)$ as $\nabla{}_{\X}f(\X_t,\y_t) = \sum_{i=1}^n\lambda_i\u_i\u_i^{\top}$, where the eigenvalues are ordered in non-increasing order. We now observe that since $\u_n$ is the leading eigenvector of $-\nabla_{\X}f(\X_t,\y_t)$, and since by Lemma \ref{lem:optSolStructMix}, $\X^*$ is a rank-one matrix which corresponds to the leading eigenvector of $-\nabla_{\X}f(\X^*,\y^*)$, then under the gap assumption (Assumption \ref{ass:gapMix}), and using the Davis-Kahan $\sin\theta$ theorem (see for instance Theorem 4 in \cite{GH15}), we have that
\begin{align}\label{eq:mainMix:4}
\Vert{\X^*-\X_t}\Vert_F &\leq \Vert{\u_n\u_n^{\top}-\X_t}\Vert_F + \Vert{\u_n\u_n^{\top}-\X^*}\Vert_F \nonumber \\
&\leq  \Vert{\u_n\u_n^{\top}-\X_t}\Vert_F + 2\sqrt{2}\frac{\Vert{\nabla_{\X}f(\X_t,\y_t)-\nabla_{\X}{}f(\X^*,\y^*)}\Vert_F}{\delta}.
\end{align}

Using Lemma \ref{lem:gradDistMix} we have
\begin{eqnarray*}
\Vert{\X^*-\X_t}\Vert_F \leq \Vert{\u_n\u_n^{\top}-\X_t}\Vert_F + C_0\sqrt{f(\X_t,\y_t)-f^*}.
\end{eqnarray*}
for $C_0 = \frac{4\Vert{\mA}\Vert\beta_g}{\sqrt{\alpha_g}\delta}$.

Plugging into \eqref{eq:mainMix:1} and using $(a+b)^2 \leq 2a^2+2b^2$, we have that
\begin{align}\label{eq:mainMix:2}
\Vert{\y^*-\y_t}\Vert_2^2 \leq C_1\Vert{\u_n\u_n^{\top}-\X_t}\Vert_F^2 + C_2\left({f(\X_t,\y_t)-f^*}\right),
\end{align}
for $C_1 = 4\Vert{\mA}\Vert^2$ and $C_2 = \frac{4}{\alpha_g} + 4\Vert{\mA}\Vert^2C_0^2 = \frac{4}{\alpha_g} + \frac{64\Vert{\mA}\Vert^4\beta_g^2}{\alpha_g\delta^2}$. 

Note that since $t\geq T_0$, using \eqref{eq:mainMix}, similarly to \eqref{eq:main:1}, it follows that $\lambda_{n-1}-\lambda_n \geq \frac{\delta}{3}$. Hence, the FW linear subproblem admits a unique optimal solution, and we can substitute $\v_t$ with $\u_n$ --- the leading eigenvector of $-\nabla{}f(\X_t,\y_t)$. Recall also that $\Vert{\u_n\u_n^{\top}-\X_t}\Vert_F^2 \leq 2(1-\u_n^{\top}\X_t\u_n)$. Thus, plugging-back into \eqref{eq:mainMix:0}, we have that for any $(\eta_{\X},\eta_{\y})\in[0,1]\times[0,1]$,
\begin{align}\label{eq:mainMix:3}
f(\X_{t+1},\y_{t+1}) &\leq f(\X_t,\y_t) + \eta_{\X}\langle{\u_n\u_n^{\top}-\X_t,\nabla_{\X}f(\X_t,\y_{t})}\rangle \nonumber \\
&+ 2\beta_g(1-\u_n^{\top}\X_t\u_n)(\Vert{\mA}\Vert^2\eta_{\X}^2 + \eta_{\y}^2C_1) \nonumber \\
& + \eta_{\y}(\y^*-\y_t)^{\top}\nabla_{\y}f(\X_t,\y_t)  + \beta_g\eta_{\y}^2C_2\left({f(\X_t,\y_t)-f^*}\right).
\end{align}

We now consider two cases. If $(1-\u_n^{\top}\X_t\u_n) \leq C_3(f(\X_t,\y_t)-f^*)$, for some $C_3>0$ to be determined later on, then, letting $\eta_{\X}=\eta_{\y}=\eta$ and using the convexity of $f(\cdot,\cdot)$, we have that for any $\eta\in[0,1]$ it holds that 
\begin{align*}
f(\X_{t+1},\y_{t+1}) - f^* &\leq \left({f(\X_t,\y_t)-f^*}\right)\left({1 - \eta + \beta_g\eta^2(2\Vert{\mA}\Vert^2C_3 +2C_1C_3+C_2)}\right) \\
&= \left({f(\X_t,\y_t)-f^*}\right)\left({1 - \eta + \eta^2\beta_gC_4}\right),
\end{align*}
where we define $C_4 = 2\Vert{\mA}\Vert^2C_3 +2C_1C_3+C_2 = 10\Vert{\mA}\Vert^2C_3 + \frac{4}{\alpha_g} + \frac{64\Vert{\mA}\Vert^4\beta_g^2}{\alpha_g\delta^2}$. 

If $2\beta_gC_4 > 1$, then taking $\eta = \frac{1}{2\beta_gC_4}$, we get
\begin{eqnarray*}
h_{t+1} \leq h_t\left({1-\frac{1}{4\beta_gC_4}}\right).
\end{eqnarray*}
Else, taking $\eta=1$ (and recalling that $1/2 \geq \beta_gC_4$) we obtain
\begin{eqnarray*}
h_{t+1} \leq \frac{h_t}{2}.
\end{eqnarray*}

In the second case ($(1-\u_n^{\top}\X_t\u_n) > C_3h_t$),  setting $\eta_{\y}=0$ in \eqref{eq:mainMix:3} we have
\begin{align*}
f(\X_{t+1},\y_{t+1})-f^* &\leq f(\X_t,\y_t)-f^* + \eta_{\X}\langle{\u_n\u_n^{\top}-\X_t,\nabla_{\X}f(\X_t,\y_{t})}\rangle\\
&+ 2\eta_{\X}^2\beta_g\Vert{\mA}\Vert^2(1-\u_n^{\top}\X_t\u_n).
\end{align*}

Using Eq. \eqref{eq:genQG:res1} from Lemma \ref{lem:genQG} w.r.t. the function $w(\X) := f(\X,\y_t)$ and with $\Y=\X=\X_t$, and recalling that according to Eq. \eqref{eq:mainMix}, $\lambda_{n-1}(\nabla{}w(\X_t)) - \lambda_n(\nabla{}w(\X_t)) \geq \frac{\delta}{3}$ (see similar calculation in \eqref{eq:main:1}), we have that
\begin{eqnarray*}
f(\X_{t+1},\y_{t+1})-f^* \leq f(\X_t,\y_t)-f^* - \eta_{\X}(1-\u_n^{\top}\X_t\u_n)\left({\frac{\delta}{3} - 2\eta_{\X}\beta_g\Vert{\mA}\Vert^2}\right).
\end{eqnarray*}

Thus, for any $\eta_{\X} \leq \frac{\delta}{6\beta_g\Vert{\mA}\Vert^2}$ (recalling $(1-\u_n^{\top}\X_t\u_n) > C_3h_t$) we have that
\begin{align*}
f(\X_{t+1},\y_{t+1})-f^* \leq f(\X_t,\y_t)-f^* - \eta_{\X}\left({\frac{\delta}{3} - 2\eta_{\X}\beta_g\Vert{\mA}\Vert^2}\right)C_3h_t.
\end{align*}

In particular, if $\frac{\delta}{12\beta_g\Vert{\mA}\Vert^2} \leq 1$, setting $\eta _{\X} = \frac{\delta}{12\beta_g\Vert{\mA}\Vert^2}$ we obtain
\begin{align*}
f(\X_{t+1},\y_{t+1})-f^* \leq f(\X_t,\y_t)-f^* - \frac{\delta^2C_3h_t}{72\beta_g\Vert{\mA}\Vert^2}.
\end{align*}

Else, setting $\eta_{\X}=1$ (and recalling $\delta/6 \geq 2\beta_g\Vert{\mA}\Vert^2$) we have
\begin{align*}
f(\X_{t+1},\y_{t+1})-f^* \leq f(\X_t,\y_t)-f^* - \frac{C_3\delta{}h_t}{6}.
\end{align*}



Thus, considering all four cases, we have that
\begin{align*}
h_{t+1} &\leq h_t\left({1 - \min\{\frac{1}{2},~\frac{1}{4\beta_gC_4},~\frac{C_3\delta^2}{72\beta_g\Vert{\mA}\Vert^2},~\frac{C_3\delta}{6}\}}\right) \\
&=h_t\left({1 - \min\{\frac{1}{2},~\frac{1}{4\beta_g\left({10\Vert{\mA}\Vert^2C_3+\frac{4\delta^2+64\Vert{\mA}\Vert^4\beta_g^2}{\alpha_g\delta^2}}\right)},~\frac{C_3\delta^2}{72\beta_g\Vert{\mA}\Vert^2},~\frac{C_3\delta}{6}\}}\right).
\end{align*}

Choosing for instance $C_3 = 1/\delta$ we get
\begin{align*}
h_{t+1} &\leq h_t\left({1 - \min\{\frac{1}{6},~\frac{\alpha_g\delta^2}{4\beta_g\left({10\alpha_g\delta\Vert{\mA}\Vert^2+4\delta^2+64\Vert{\mA}\Vert^4\beta_g^2}\right)},~\frac{\delta}{72\beta_g\Vert{\mA}\Vert^2}\}}\right).
\end{align*}

Finally, we turn to prove the third part of the theorem (Eq. \eqref{eq:mainMix:res3s}). Applying the Davis-Kahan $\sin\theta$ theorem in the same way as in the derivation of Eq. \eqref{eq:mainMix:4} above, and recalling that w.l.o.g. $\v_t = \u_n$, we have that
\begin{align*}
\Vert{\v_t\v_t^{\top}-\X^*}\Vert_F^2 &\leq  8\frac{\Vert{\nabla_{\X}f(\X_t,\y_t)-\nabla_{\X}{}f(\X^*,\y^*)}\Vert_F^2}{\delta^2} \leq \frac{16\beta_g^2\Vert{\mA}\Vert^2}{\alpha_g\delta^2}h_t,
\end{align*}
were the last inequality follows from plugging-in the bound in Lemma \ref{lem:gradDistMix}.
\end{proof}

\subsubsection{Verifying Assumption \ref{ass:gapMix}}

Similarly to Section \ref{sec:verify}, for Problem \eqref{eq:optProbMix} we can also suggest a simple procedure for practical verification of a lower-bound for the parameter $\delta$ in Assumption \ref{ass:gapMix} (provided it is greater than zero).

If we let $(\X^*,\y^*)$ denote an optimal solution to Problem \eqref{eq:optProbMix}, then combining Weyl's inequality for the eigenvalues and Lemma \ref{lem:gradDistMix}, we have that for any $t\geq 1$, the pair $(\X_t,\y_t)$ from Algorithm \ref{alg:mixed}, satisfies for all $i\in[n]$:
\begin{align*}
\vert{\lambda_i(\nabla_{\X}f(\X_t,\y_t)) - \lambda_i(\nabla_{\X}f(\X^*,\y^*))}\vert \leq \sqrt{\frac{2\Vert{\mA}\Vert^2\beta_g^2}{\alpha_g}(f(\X_t,\y_t)-f(\X^*,\y^*))}.
\end{align*}
Using the short notation $\textsc{gap}(\X,\y) = \lambda_{n-1}(\nabla_{\X}f(\X,\y)) - \lambda_n(\nabla_{\X}f(\X,\y))$, the above implies that
\begin{align}\label{eq:eigengapEstMix}
&\vert{\textsc{gap}(\X^*,\y^*) - \textsc{gap}(\X_t,\y_t)}\vert \leq 2 \sqrt{\frac{2\Vert{\mA}\Vert^2\beta_g^2}{\alpha_g}(f(\X_t,\y_t)-f(\X^*,\y^*))}.
\end{align}

Thus, as discussed in Section  \ref{sec:verify}, if  Assumption \ref{ass:gapMix} indeed holds with parameter $\delta > 0$, then using Eq. \eqref{eq:eigengapEstMix}, once we arrive at an iteration $t$ for which it holds that the RHS of \eqref{eq:eigengapEstMix} is smaller for instance than $\delta/3$, by computing the eigenvalues $\lambda_{n-1}(\nabla_{\X}f(\X_t,\y_t)), \lambda_n(\nabla_{\X}f(\X_t,\y_t))$, we can verify that Assumption \ref{ass:gapMix} holds with parameter at least $\delta/3$, and in particular that there exists a unique optimal solution and that its low-rank matrix component is indeed rank-one.

Finally, note that since $f$ is convex, we have that
\begin{align}\label{eq:verifyerrormix}
f(\X_t,\y_t)-f(\X^*,\y^*) &\leq \langle{\X_t-\X^*,\nabla_{\X}f(\X_t,\y_t)}\rangle + \langle{\y_t-\y^*,\nabla_{\y}f(\X_t,\y_t)}\rangle \nonumber \\
&\leq \langle{\X_t-\v_t\v_t^{\top},\nabla_{\X}f(\X_t,\y_t)}\rangle \nonumber \\
&~+ \langle{\y_t, \nabla_{\y}f(\X_t,\y_t)}\rangle - \min_{\u\in\mK}\langle{\u,\nabla_{\y}f(\X_t,\y_t)}\rangle,
\end{align}
where $\v_t$ is the eigenvector computed on iteration $t$ of Algorithm \ref{alg:mixed}. 

Thus, when linear minimization over $\mK$ is efficient, we can upper-bound the approximation error in the RHS of \eqref{eq:eigengapEstMix} with the RHS of \eqref{eq:verifyerrormix} which is efficient to compute.

\section{Extension to Nonsmooth Functions}\label{sec:nonsmooth}

We now consider an extension of our results to the case in which $f(\cdot)$ is convex over $\mbS^n$ but not smooth. For instance, as an example, two applications of interest in the context of rank-one matrix recovery are $f(\X) := \Vert{\X-\M}\Vert_1$, which is also a popular formulation of the Robust-PCA problem (here $\M$ is the observed data), and $f(\X) := \frac{1}{2}\Vert{\X-\M}\Vert_F^2 +\lambda\Vert{\X}\Vert_1$, which is useful when attempting to recover a matrix $\X$ that is both low-rank and sparse from the noisy observation $\M$ (e.g., \cite{Richard12,Garber18b}).

Towards this end, we recall the following sufficient and necessary optimality condition for constrained nonsmooth convex optimization.
\begin{lemma}[Corollary 3.68 in \cite{Beck17}]\label{lem:subgrad}
$\X^*\in\mS_n$ is an optimal solution of \eqref{eq:optProb} (even when $f$ is nonsmooth) if and only if there exists $\G^*\in\partial{}f(\X^*)$ such that 
\begin{align}\label{eq:nonsmooth:subgrad}
\forall\X\in\mS_n:\qquad \langle{\X-\X^*,\G^*}\rangle \geq 0.
\end{align}
\end{lemma}

The following assumption extends Assumption \ref{ass:gap} to nonsmooth functions.

\begin{assumption}\label{ass:gap:NS}
There exists an optimal solution $\X^*$ to Problem \eqref{eq:optProb} such that $\lambda_{n-1}(\G^*)-\lambda_n(\G^*)  =\delta > 0$, where $\G^*$ is a subgradient of $f(\cdot)$ at $\X^*$ for which Eq. \eqref{eq:nonsmooth:subgrad} holds.
\end{assumption}

\begin{lemma}\label{lem:QG:NS}
Suppose Assumption \ref{ass:gap:NS} holds for some optimal solution $\X^*\in\mS_n$. Then, $\X^*$ is both the unique optimal solution to Problem \eqref{eq:optProb}, and rank-one. Moreover, Problem \eqref{eq:optProb} has the quadratic growth property
\begin{align*}
\forall \X\in\mS_n: \quad \Vert{\X-\X^*}\Vert_F^2 \leq \frac{2}{\delta}\left({f(\X)- f^*)}\right),
\end{align*}
even though $f(\cdot)$ is nonsmooth.
\end{lemma}

\begin{proof}
From Lemma \ref{lem:subgrad} if follows that under Assumption \ref{ass:gap:NS}, $\X^*$ must be the unique rank-one matrix corresponding to the eigenvector of $\G^*$ with smallest eigenvalue (where $\G^*$ is the subgradeint defined in Lemma \ref{lem:subgrad}), since otherwise, letting $\u_n^*$ denote the eigenvector of $\G^*$ corresponding to the smallest eigenvalue, we will have that $\langle{\u_n^*\u_n^{*\top}-\X^*,\G^*}\rangle < 0$, which contradicts the optimality of $\X^*$. 

Using the above, the quadratic growth property follows from repeating the steps of the proof of Eq. \eqref{eq:genQG:res2} in Lemma \ref{lem:genQG}, replacing $\nabla{}f(\X^*)$ with $\G^*$.
\end{proof}

Towards applying Frank-Wolfe-type methods to Problem \eqref{eq:optProb} with nonsmooth $f$, we will consider a standard approach of replacing the nonsmooth $f(\cdot)$ with a smooth approximation. 
\begin{definition}
We say a convex function $f_{(\alpha,\beta)}:\mbS^n\rightarrow\reals$ is an $(\alpha,\beta)$-smooth approximation of a convex function $f:\mbS^n\rightarrow\reals$, if i) for all $\X\in\mbS^n$: $\vert{f(\X) - f_{(\alpha,\beta)}(\X)}\vert \leq \alpha$, and ii) $f_{(\alpha,\beta)}$ is $\beta$-smooth.
\end{definition}
We refer the interested reader to \cite{Beck12} for an in-depth treatment of the subject of constructing smooth approximations with many important examples.

We note that typically $\beta$ scales with $1/\alpha$. In particular, usually $\alpha$ is chosen so that $\alpha = O(\epsilon)$, where $\epsilon$ is the target approximation-error desired, which causes $\beta$ to be of the order $\beta = O(1/\epsilon)$. Note however that since, as discussed, the smoothness parameter will typically scale with $1/\epsilon$, the results in Theorems \ref{thm:rankoneFW} and \ref{thm:rankoneFWimprov}, when applied to the smooth approximation $f_{(\alpha,\beta)}$, give fast convergence rates only after roughly $O(\beta^3) = O(1/\epsilon^3)$ initial iterations. Since applying the standard Frank-Wolfe convergence result to $f_{(\alpha,\beta)}(\cdot)$ will already result in a $O(1/\epsilon^2)$ rate, these fast rate results become meaningless. We thus consider only adapting the result of Theorem \ref{thm:regFW}, which does not have a ``burn-in'' phase, but does require an estimate of the gap $\delta$.

\begin{theorem}\label{thm:regFW_NS}
Under Assumption \ref{ass:gap:NS}, the iterates of Algorithm \ref{alg:regfw}, when applied to an $(\alpha,\beta)$-smooth approximation $f_{(\alpha,\beta)}$ of $f$,  and with gap estimate $\hat{\delta}$ such that $0 < \hat{\delta} \leq \delta$ (where $\delta$ is as defined in Assumption \ref{ass:gap:NS}), satisfy
\begin{align*}
\forall t\geq 0: \quad f(\X_{t+1})-f^* \leq \left({f(\X_1)-f^*}\right)\exp\Big({-\min\{\frac{1}{2},\frac{\hat{\delta}}{4\beta}\}t}\Big) + O(\alpha).
\end{align*}
\end{theorem}

Indeed, we see that in the typical case $\frac{\hat{\delta}}{4\beta}\leq\frac{1}{2}$, and when $\alpha = O(\epsilon)$, $\beta = O(1/\epsilon)$, the number of iterations to reach $O(\epsilon)$ approximation error is of the order $O\left({\frac{\log{1/\epsilon}}{\hat{\delta}\epsilon}}\right)$, which up to a $\log{1/\epsilon}$ factor, is what we expect when optimizing a nonsmooth $\hat{\delta}$-strongly convex function.

\begin{proof}
The proof follows from simple modifications of the proof of Theorem \ref{thm:regFW}, as we now detail. Let us denote $f_{(\alpha,\beta)}^* = \min_{\X\in\mS_n}f_{(\alpha,\beta)}(\X)$. On any iteration $t$ it holds that
\begin{align*}
f_{(\alpha,\beta)}(\X_{t+1}) - f_{(\alpha,\beta)}^* &\underset{(a)}{\leq} f_{(\alpha,\beta)}(\X_t) -f_{(\alpha,\beta)}^* + \eta\langle{\v_t\v_t^{\top}-\X_t,\nabla{}f_{(\alpha,\beta)}(\X_t)}\rangle \\
&+ \frac{\eta^2\beta}{2}\Vert{\v_t\v_t^{\top}-\X_t}\Vert_F^2 \\
&\underset{(b)}{\leq} f_{(\alpha,\beta)}(\X_t) -f_{(\alpha,\beta)}^* + \eta\langle{\X^*-\X_t,\nabla{}f_{(\alpha,\beta)}(\X_t)}\rangle \\
&+ \frac{\eta^2\beta}{2}\Vert{\X^*-\X_t}\Vert_F^2 \\
&\underset{(c)}{\leq} f_{(\alpha,\beta)}(\X_t) -f_{(\alpha,\beta)}^* - \eta\left({f_{(\alpha,\beta)}(\X_t) -f_{(\alpha,\beta)}(\X^*)}\right)  \\
&+ \frac{\eta^2\beta}{\hat{\delta}}\left({f(\X_t)-f(\X^*)}\right),
\end{align*}
where (a) follows from the $\beta$-smoothness of $f_{(\alpha,\beta)}$, (b) follows from the optimal choice of $\v_t$ and since, under Assumption \ref{ass:gap:NS}, $\X^*$ is rank-one, and (c) follows from the convexity of $f_{(\alpha,\beta)}$, Lemma \ref{lem:QG:NS}, and since  $\hat{\delta} \leq \delta$.

Let $\Z^*$ be the minimizer of $f_{(\alpha,\beta)}$ over $\mS_n$. 
Since $f_{(\alpha,\beta)}^*= f_{(\alpha,\beta)}(\Z^*) \geq f(\Z^*) - \alpha \geq f(\X^*) - \alpha \geq f_{(\alpha,\beta)}(\X^*) - 2\alpha$ and $f(\X^*) \geq f_{(\alpha,\beta)}(\X^*)-\alpha \geq f_{(\alpha,\beta)}^*-\alpha$,  the above leads to
\begin{align*}
f_{(\alpha,\beta)}(\X_{t+1})-f_{(\alpha,\beta)}^* &\leq \left({f_{(\alpha,\beta)}(\X_t)-f_{(\alpha,\beta)}^*}\right)\Big({1-\eta + \frac{\eta^2\beta}{\hat{\delta}}}\Big)\\
& + \alpha(2\eta + 2\eta^2\beta/\hat{\delta}).
\end{align*}
Plugging-in the value of $\eta$ and considering the two possible cases as in the proof of Theorem \ref{thm:regFW}, we have
\begin{align*}
f_{(\alpha,\beta)}(\X_{t+1})-f_{(\alpha,\beta)}^* &\leq \left({f_{(\alpha,\beta)}(\X_t)-f_{(\alpha,\beta)}^*}\right)\Big({1-\min\{\frac{1}{2},\frac{\hat{\delta}}{4\beta}\}}\Big) \\
&+ 3\alpha\min\{1, \frac{\hat{\delta}}{2\beta}\}.
\end{align*}
Unrolling the recursion, using $1-x \leq e^{-x}$, and the formula for the sum of an infinite converging geometric series, we get 
\begin{align*}
f_{(\alpha,\beta)}(\X_{t+1})-f_{(\alpha,\beta)}^* \leq \left({f_{(\alpha,\beta)}(\X_1)-f_{(\alpha,\beta)}^*}\right)\exp\Big({-\min\{\frac{1}{2},\frac{\hat{\delta}}{4\beta}\}t}\Big) + O(\alpha).
\end{align*}

Finally, replacing $f_{(\alpha,\beta)}(\cdot)$ with $f(\cdot)$, we have that
\begin{align*}
f(\X_{t+1})-f^* \leq \left({f(\X_1)-f^*}\right)\exp\Big({-\min\{\frac{1}{2},\frac{\hat{\delta}}{4\beta}\}t}\Big) + O(\alpha).
\end{align*}

\end{proof}

\section{Numerical Experiments}\label{sec:exp}

In this section we bring empirical evidence in support of our main assumption, Assumption \ref{ass:gap} (and the closely-related Assumption \ref{ass:gapMix}), and some empirical comparison between the various methods considered in this work.

\subsection{Empirical evidence for gap assumption}
We consider two tasks, one of recovering a rank-one matrix from quadratic measurements, a problem closely related to phase-retrieval (for which the underlying assumption is Assumption \ref{ass:gap}), and rank-one robust PCA  (for which the underlying assumption is Assumption \ref{ass:gapMix}). In both cases we construct synthetic random instances of the problems and demonstrate that i) the proposed models indeed recover the ground-truth signal with low error, and ii) the data indeed satisfy the gap assumption. 

\paragraph{Rank-one recovery from quadratic measurements:}
 We let $\x_0=\sqrt{n}\v_0$, where $\v_0\in\reals^n$ is a random unit vector, and we draw $m$ pairs of random unit vectors $\{(\a_i,\b_i)\}_{i=1}^m\subset\reals^n\times\reals^n$. The vector of quadratic measurements of $\x_0$ is given by $\y_0(i)= \a_i^{\top}\x_0\x_0^{\top}\b_i$, $\y_0\in\reals^m$, and the observed noisy vector is given by $\y = \y_0 + \sqrt{c}\n$, where $\n\in\reals^m$ is a vector with standard Gaussian entries.
The goal is to recover the rank-one matrix $\x_0\x_0^{\top}$ from the noisy measurements vector $\y$, and towards this we consider the problem
\begin{align}\label{eq:quadSenseProb}
\min_{\X\succeq 0~\trace(\X)=\tau}\{f(\X) := \frac{1}{2}\sum_{i=1}^m\left({\a_i^{\top}\X\b_i - \y(i)}\right)^2\}.
\end{align}

We solve Problem \eqref{eq:quadSenseProb} to high accuracy (approximation error w.r.t. function value less than 1e-12 in our MATLAB implementation\footnote{the bound on the approximation error is verified by computing the duality gap, which is an upper-bound on the approximation error w.r.t. the function value (see for instance \cite{Jaggi13b})}) using the standard Frank-Wolfe method (Algorithm \ref{alg:fw}), and we denote the found solution by $\X^*$. We produce our estimate for the ground-truth vector $\x_0$ by computing the leading eigenvector of the matrix $-\nabla{}f(\X^*)$, which we denote by $\v^*$, and scaling it to have the same norm as $\x_0$, i.e., we take the vector $\sqrt{n}\v^*$. Note that estimation based on the eigenvector $\v^*$ is motivated by Eq. \eqref{eq:main:res3} (in particular, since $\X^*$ is only a high-accuracy approximated solution, it may not be rank-one). We measure the relative recovery error by $\Vert{n\v^*\v^{*\top}-\x_0\x_0^{\top}}\Vert_F^2/\Vert{\x_0\x_0^{\top}}\Vert_F^2 = \frac{1}{n^2}\Vert{n\v^*\v^{*\top}-\x_0\x_0^{\top}}\Vert_F^2$. 


In our experiments we set $m=20n$, $\tau = 0.5n$, and the noise parameter $c$ is set to either $0.5$ or $1.5$. We note that we choose the trace bound $\tau$ strictly smaller than $\trace(\x_0\x_0^{\top})=n$, since otherwise the optimal solution will naturally also fit some of the noise and will result in a higher-rank matrix. All results are averaged over 20 i.i.d runs. The results are presented in Table \ref{table:quadrecoverExp}.

As it can be seen in Table \ref{table:quadrecoverExp}, all random instances indeed satisfy Assumption \ref{ass:gap} with substantial eigen-gap. Moreover, the gap does not vary much with the dimension. We note that even though $\X^*$ is only a high-accuracy approximated solution to  \eqref{eq:quadSenseProb} (approximation error $<$ 1e-12), since Problem \eqref{eq:quadSenseProb} satisfies Assumption \ref{ass:funcStruct} (i.e., it can be written as $f(\X)=g(\mA\X)$ with $g(\x) = \frac{1}{2}\Vert{\x-\y}\Vert_2^2$, and so $\alpha_g=\beta_g=1$), a simple calculation using Eq. \eqref{eq:eigengapEst}, and recalling that $\a_i,\b_i, i=1,\dots,m$ are all unit vectors, implies that the eigen-gap estimates in Table \ref{table:quadrecoverExp} represent, up to negligible error, the eigen-gaps in the gradient vector at the exact optimal solution. Using Lemma \ref{lem:optSolStruct}, this in turn verifies that for all random instances it holds that there is a unique optimal solution and that it is indeed rank-one.

\begin{table*}\renewcommand{\arraystretch}{1.3}
{\footnotesize
\begin{center}
  \begin{tabular}{| c | c | c | c | c | c |} \hline
    noise level ($c$) &dimension  ($n$)& avg. recovery error & min/avg. gap in $\nabla{}f(\X^*)$  & avg. SNR \\ \hline
0.5 & 100 & 0.0638 & 2.9730 / 4.5488  & 1.9931\\ \hline
0.5 & 200 & 0.0621 & 3.7889 / 4.3656 & 1.9935\\ \hline
0.5 & 400 & 0.0625 & 3.8897 / 4.3656 & 2.0053\\ \hline
0.5 & 600 & 0.0623 & 3.9671 / 4.3927 & 2.0141\\ \hline
1.5 & 100 & 0.1146 & 1.1551 / 2.3836  & 0.6736\\ \hline
1.5 &200 & 0.1129 & 1.5993 / 1.9936 & 0.6735\\ \hline
1.5 &400 & 0.1142 & 1.4172 / 1.9756 & 0.6547\\ \hline
1.5 &600 & 0.1143 & 1.1452 / 1.9320 & 0.6582\\ \hline
  \end{tabular}
\caption{Results for Problem \eqref{eq:quadSenseProb}. The recovery error is given by $\Vert{n\v^*\v^{*\top}-\x_0\x_0^{\top}}\Vert_F^2/\Vert{\x_0\x_0^{\top}}\Vert_F^2$ ($\v^*$ is leading eigenvector of $-\nabla{}f(\X^*)$), the gap in $\nabla{}f(\X^*)$ is given by $\lambda_{n-1}(\nabla{}f(\X^*)) - \lambda_n(\nabla{}f(\X^*))$, and the signal-to-noise ratio (SNR) is given by $\Vert{\y_0}\Vert^2/\Vert{\sqrt{c}\n}\Vert^2$.}
  \label{table:quadrecoverExp}
\end{center}
}
\end{table*}\renewcommand{\arraystretch}{1}

\paragraph{Rank-one Robust PCA:} We consider the task of extracting a rank-one matrix from its sparsely-corrupted observation. We let $\M = \x_0\x_0^{\top} + \frac{1}{2}(\Y_0+\Y_0^{\top})$, where $\x_0\in\reals^n$ is a random unit vector ($\x_0\x_0^{\top}$ is the rank-one matrix to recover), and $\Y_0$ is sparse, with each entry being either $1$ or $-1$ with probability $p$ and zero otherwise ($p<< 1$).
 Towards recovering $\X_0 = \x_0\x_0^{\top}$, we consider the optimization problem
\begin{align}\label{eq:RPCAexp}
\min_{\X\succeq 0~\trace(\X)=\tau,~\Y:\Vert{\Y}\Vert_1 \leq s}\{f(\X,\Y) := \frac{1}{2}\Vert{\X+\Y-\M}\Vert_F^2\}.
\end{align}

Similarly to the previous example, we solve Problem \eqref{eq:RPCAexp} to high accuracy (approximation error w.r.t. function value less than 1e-12) using our Algorithm \ref{alg:mixed}, and we denote by $(\X^*,\Y^*)$ the obtained solution. As before, since $\X^*$ may not be rank-one, we produce our estimate for the ground-truth matrix $\X_0$ by taking the matrix $\v^*\v^{*\top}$, where $\v^*$ is the leading eigenvector of $-\nabla_{\X}f(\X^*,\y^*)$ (note this is motivated by Eq. \eqref{eq:mainMix:res3s}), and we measure the recovery error by $\Vert{\v^*\v^{*\top}-\X_0}\Vert_F^2$.
In all experiments we set $s=0.97\cdot\Vert{\frac{1}{2}(\Y_0+\Y_0^{\top})}\Vert_1$, $\tau = 0.7$, and the noise sampling probability $p$ is either  $1/\sqrt{25n}$ or  $1/\sqrt{n}$. All results are averaged over 20 i.i.d runs. The results are presented in Table \ref{table:RPCAExp}.

As in the previous example,  Table \ref{table:quadrecoverExp} shows that all random instances indeed satisfy Assumption \ref{ass:gapMix} with substantial gap, and the gap does not change drastically with the dimension. Here also we note that even though $(\X^*,\y^*)$ is only a high-accuracy approximated solution  (approximation error $<$ 1e-12),  a simple calculation using Eq. \eqref{eq:eigengapEstMix} implies that the eigen-gap estimates in Table \ref{table:quadrecoverExp} represent, up to negligible error, the eigen-gaps in the gradient vector at the exact optimal solution. Using Lemma \ref{lem:optSolStructMix}, this verifies that for all random instances it holds that there is a unique optimal solution pair, and that the low-rank matrix component of it is indeed rank-one.

\begin{table*}\renewcommand{\arraystretch}{1.3}
{\footnotesize
\begin{center}
  \begin{tabular}{| c | c | c | c | c |} \hline
   noise prob ($p$) & dimension  ($n$) & avg. recovery error & min/avg. gap in $\nabla_{\X}{}f(\X^*)$ & avg. SNR \\ \hline
$1/\sqrt{25n}$ & 100 &0.0026 & 0.2017 / 0.2179  & 0.0098\\ \hline
$1/\sqrt{25n}$ & 200 & 0.0028 &0.2091 / 0.2169 & 0.0035\\ \hline
$1/\sqrt{25n}$ & 400 & 0.0040 &0.1995 / 0.2056 & 0.0012\\ \hline
$1/\sqrt{25n}$ & 600 & 0.0046 & 0.1966 / 0.2010  & 6.7945e-04\\ \hline
$1/\sqrt{25n}$ & 1000 & 0.0058 &0.1850 / 0.1888 & 3.1631e-04\\ \hline
$1/\sqrt{n}$ & 100 & 0.0153 &0.0996 / 0.1177  & 0.0020\\ \hline
$1/\sqrt{n}$ & 200 & 0.0178 &0.0956 / 0.1080 & 7.0348e-04\\ \hline
$1/\sqrt{n}$ & 400 & 0.0216 &0.0793 / 0.0953 & 2.4988e-04\\ \hline
$1/\sqrt{n}$ & 600 & 0.0260 &0.0724 / 0.0792 & 1.3607e-04\\ \hline
$1/\sqrt{n}$ & 1000 & 0.0323 &0.0484 / 0.0599 & 6.3072e-05\\ \hline
  \end{tabular}
\caption{Results for Problem \eqref{eq:RPCAexp}. The recovery error is given by $\Vert{\v^*\v^{*\top}-\x_0\x_0^{\top}}\Vert_F^2$ ($\v^*$ is leading eigenvector of $-\nabla_{\X}f(\X^*,\y^*)$), the gap in $\nabla_{\X}f(\X^*)$ is given by $\lambda_{n-1}(\nabla_{\X}f(\X^*)) - \lambda_n(\nabla_{\X}f(\X^*))$, and the signal-to-noise ratio (SNR) is given by $\Vert{\x_0\x_0^{\top}}\Vert_F^2/\Vert{\frac{1}{2}(\Y_0+\Y_0^{\top})}\Vert_F^2$.}
  \label{table:RPCAExp}
\end{center}
}
\end{table*}\renewcommand{\arraystretch}{1}

\subsection{Comparison of Frank-Wolfe variants}

\begin{table*}\renewcommand{\arraystretch}{1.3}
{\footnotesize
\begin{center}
  \begin{tabular}{|p{0.18\linewidth} | p{0.83\linewidth}|} \hline
   algorithm &description \\ \hline
   FW-ls(opt1) & Frank-Wolfe  with exact line search (Algorithm \ref{alg:fw} with option 1) \\ \hline
   FW-ls(opt2) & Frank-Wolfe  with line-search over quadratic upper-bound (Algorithm \ref{alg:fw} with option 2)\\ \hline
   FWPG & Frank-Wolfe +  projected gradient steps (Algorithm \ref{alg:fwpg} with option 1) \\ \hline
   RegFW-ls(opt1) & Regularized Frank-Wolfe (Algorithm \ref{alg:regfw}). After computing the eigenvector $\v_t$ on each iteration $t$, the step-size is set via exact line-search (similarly to option 1 in Algorithm \ref{alg:fw}). This does not change the theoretical convergence guarantees but significantly improves the convergence in practice. The gap estimate $\hat{\delta}$ which the algorithm requires is taken from Table \ref{table:quadrecoverExp}: we set $\hat{\delta} = 3$ when $c=0.5$ and $\hat{\delta}=1$ when $c=1.5$  \\ \hline
  \end{tabular}
\caption{Description of Frank-Wolfe variants used in the numerical comparison.}
  \label{table:algorithms}
\end{center}
}
\end{table*}\renewcommand{\arraystretch}{1}

We turn to present preliminary empirical comparison between four Frank-Wolfe variants presented, on the rank-one recovery from quadratic measurements task --- Problem \eqref{eq:quadSenseProb}, fixing the dimension to $n=200$ and setting the noise parameter $c$ to either $0.5$ or $1.5$. The tested algorithms are detailed in Table \ref{table:algorithms}.

Since all variants except for FW-ls(opt1) rely on the smoothness parameter $\beta$, we try several values. We begin with $\beta=\sqrt{n}=\sqrt{200}$ and observe that this choice seems quite conservative, and thus we also try $\beta = 1$ and $\beta = 0.1$.
All algorithms are initialized with the same matrix which is generated as follows: we pick $\x\in\reals^n$ to be a random unit-norm vector. We then set the initialization to $\X_1 \gets\arg\min_{\Y\succeq 0,\trace(\Y)=\tau}\langle{\Y,\nabla{}f(\tau\cdot\x\x^{\top})}\rangle$ \footnote{we note this is a common initialization for Frank-Wolfe, and actually is equivalent to initializing Frank-Wolfe with $\tau\cdot\x\x^{\top}$, and running for one iteration with the classical step-size rule $\eta_t = \frac{2}{t+1}$}. Note that $\X_1$ simply corresponds to computing the leading eigenvector of $-\nabla{}f(\tau\cdot\x\x^{\top})$ and returning the corresponding rank-one matrix scaled by $\tau$. The results are the average of 20 i.i.d runs.

\begin{figure}
     \centering
     \begin{subfigure}[t]{0.43\textwidth}
         \centering
         \includegraphics[width=\textwidth]{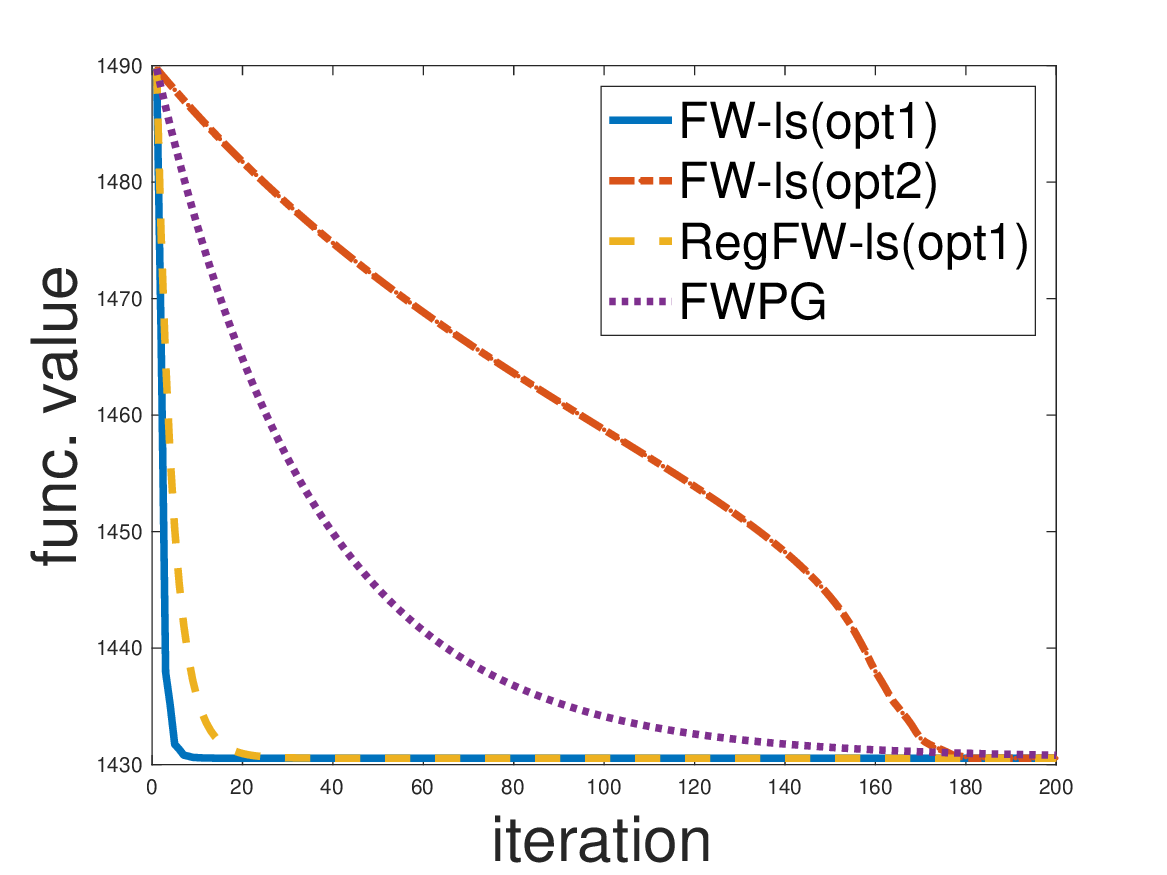}
         \caption{$c=0.5, \beta=\sqrt{200}$}
         \label{fig:y equals x}
     \end{subfigure}
     \begin{subfigure}[t]{0.43\textwidth}
         \centering
         \includegraphics[width=\textwidth]{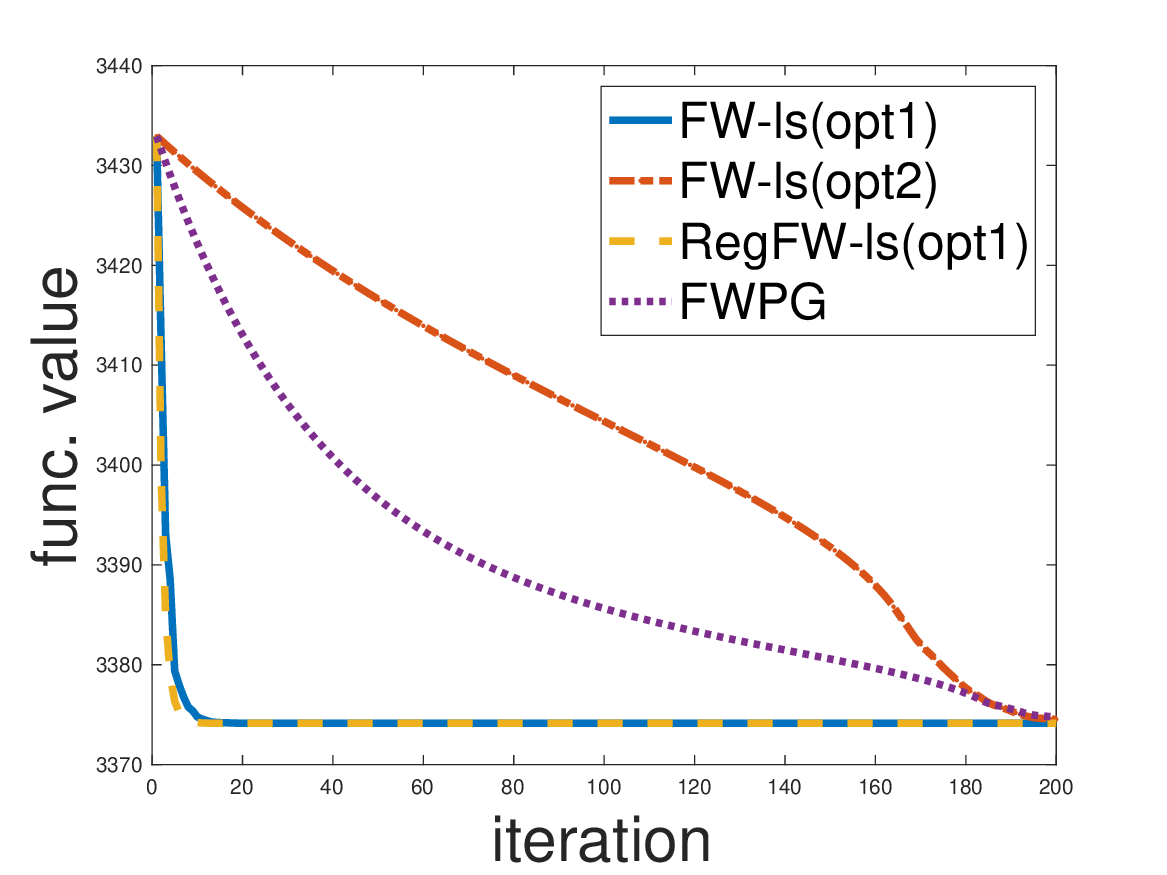}
         \caption{$c=1.5, \beta=\sqrt{200}$}
         \label{fig:three sin x}         
     \end{subfigure}
     \hfill
     
     \begin{subfigure}[t]{0.43\textwidth}
         \centering
         \includegraphics[width=\textwidth]{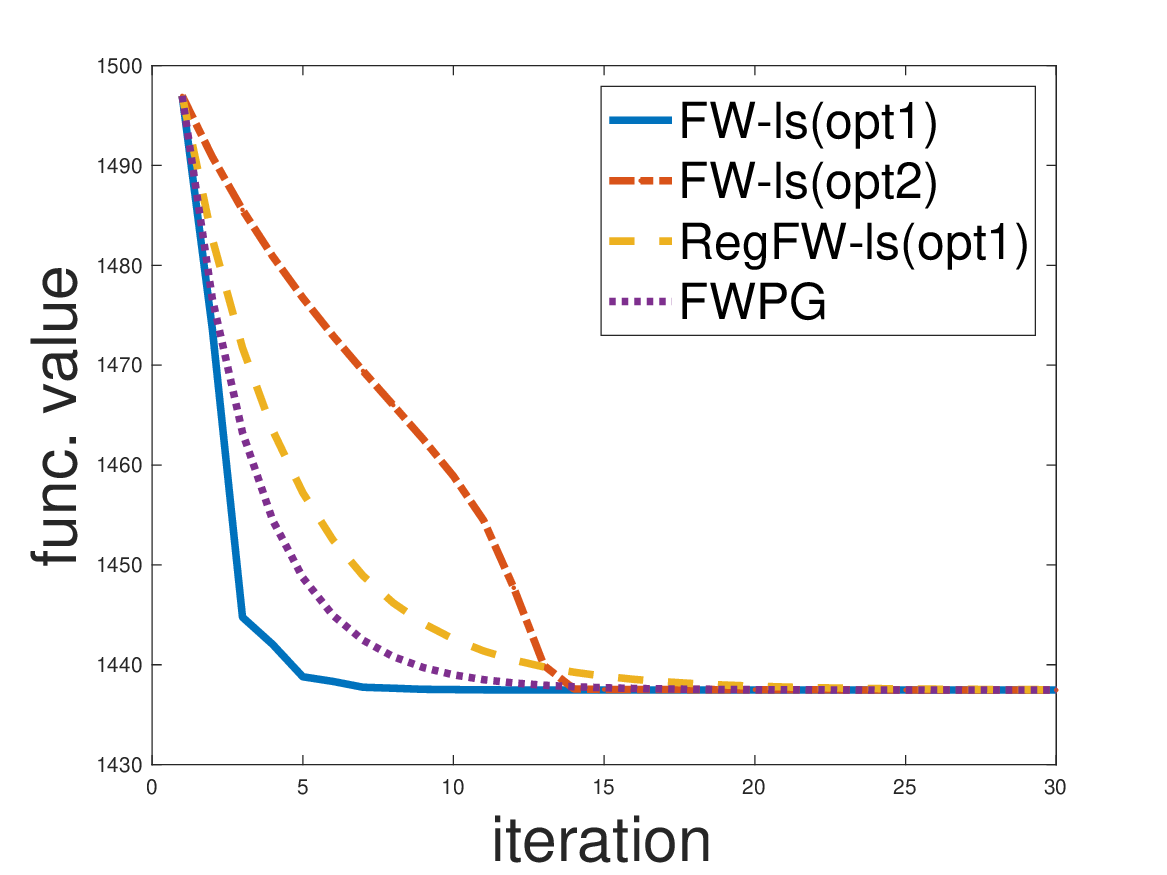}
         \caption{$c=0.5, \beta=1$}
         \label{fig:five over x}
     \end{subfigure}
     \begin{subfigure}[t]{0.43\textwidth}
         \centering
         \includegraphics[width=\textwidth]{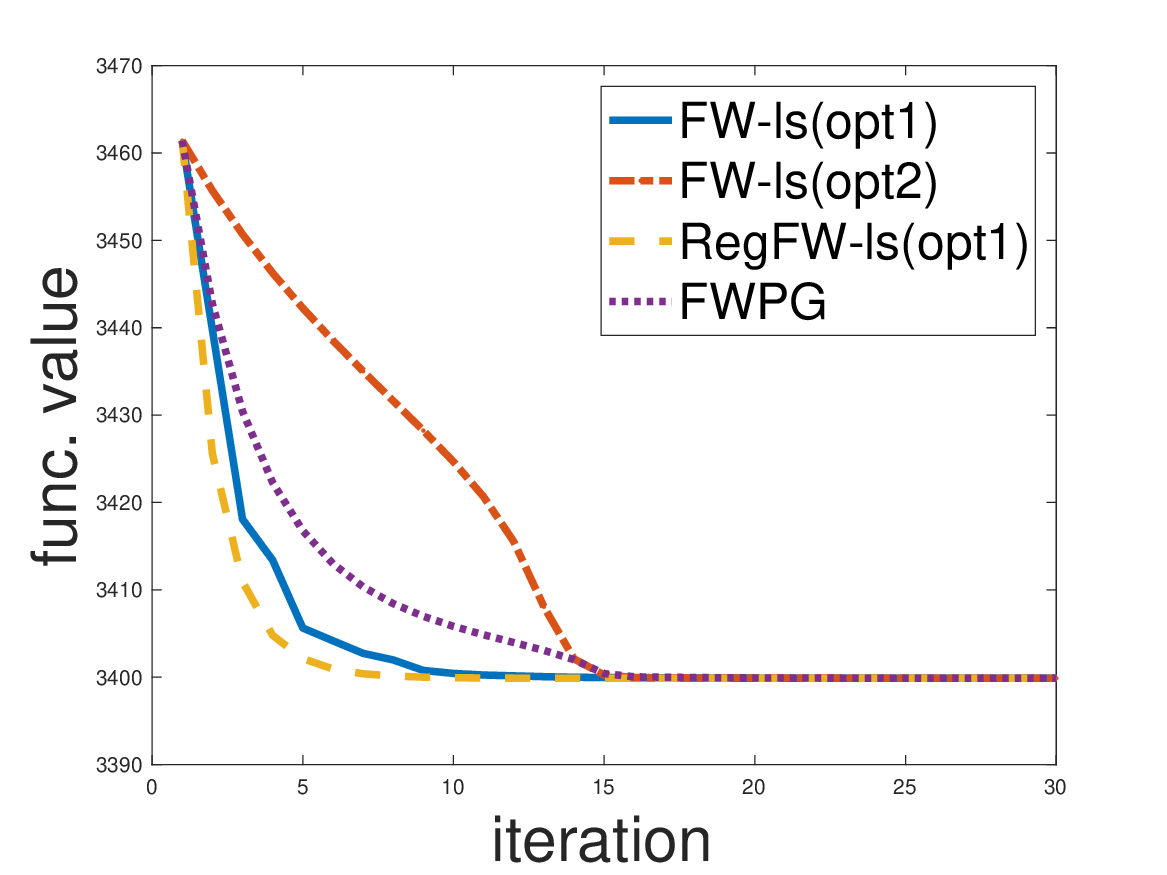}
         \caption{$c=1.5, \beta=1$}
         \label{fig:five over x}
     \end{subfigure}

     \begin{subfigure}[t]{0.43\textwidth}
         \centering
         \includegraphics[width=\textwidth]{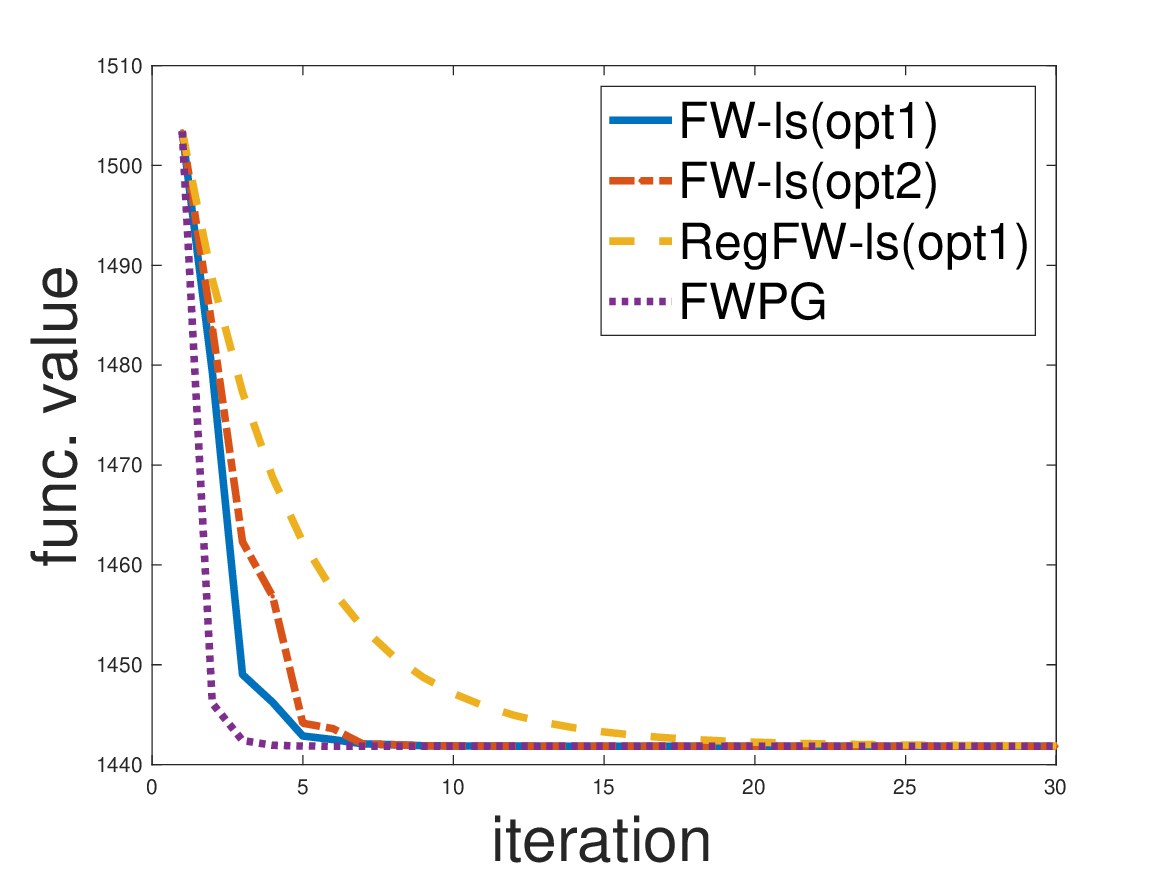}
         \caption{$c=0.5, \beta=0.1$}
         \label{fig:five over x}
     \end{subfigure}
     \begin{subfigure}[t]{0.43\textwidth}
         \centering
         \includegraphics[width=\textwidth]{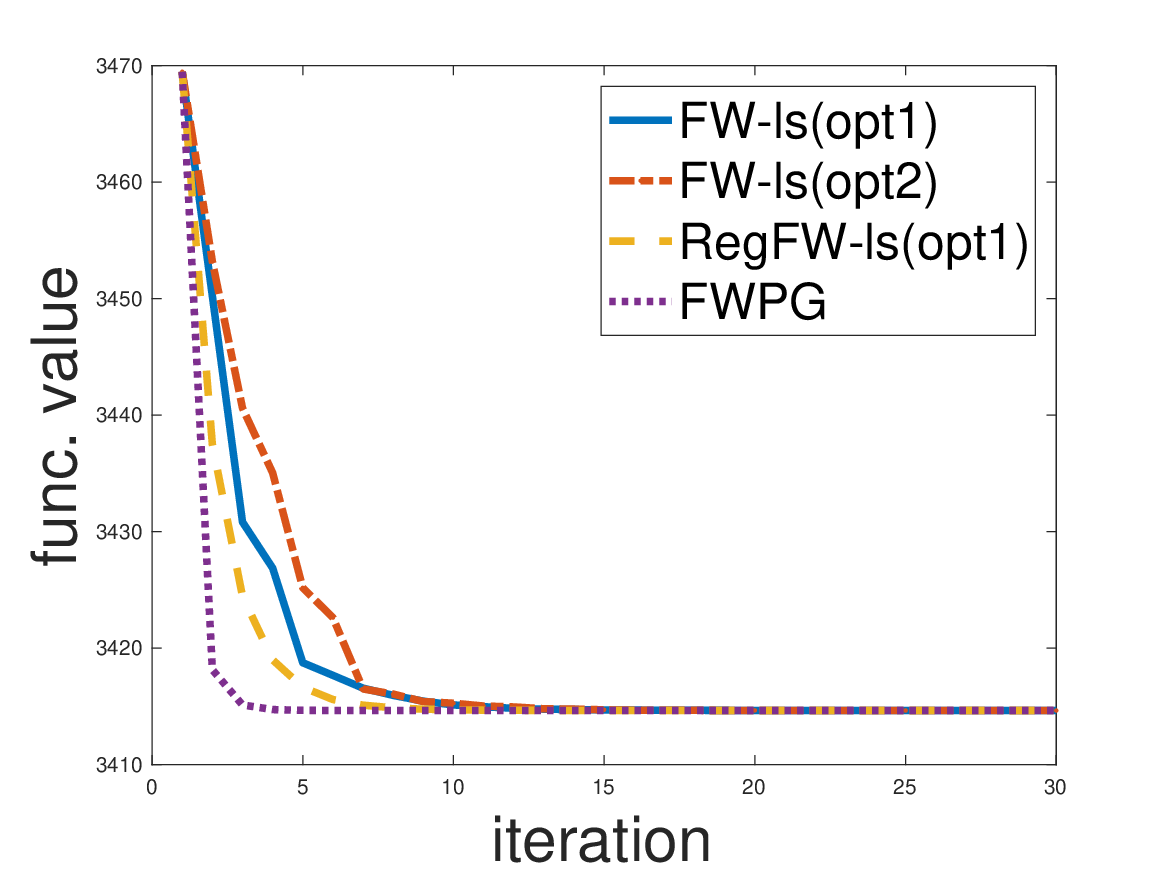}
         \caption{$c=1.5, \beta=0.1$}
         \label{fig:five over x}
     \end{subfigure}
        \caption{Comparison of Frank-Wolfe variants for rank-one matrix recovery from quadratic measurements (Problem \eqref{eq:quadSenseProb}).}
        \label{fig:compare}
\end{figure}

The results are given in Figure \ref{fig:compare}. We see that the variants RegFW-ls(opt1) and FWPG can indeed be faster than standard Frank-Wolfe with line-search (FW-ls(opt1) and FW-ls(opt2)) when tuned properly. In particular, for $\beta =0.1$ which gives the best results for all variants, we see that FWPG has the fastest convergence (with either $c=0.5$ or $c=1.5$). Importantly, when examining the rank of the iterates of FWPG, we observe that in all cases except for $(c=1.5,\beta=\sqrt{200})$ and $(c=1.5, \beta=1)$, our initialization already starts FWPG in the regime in which only projected gradient steps are used which means that FWPG only maintains a rank-one matrix throughout the run, as opposed to all other variants.


\appendix

\section{Proof of Lemma \ref{lem:gapRobust}}

The lemma is an adaptation of Lemma 3 in \cite{Garber19} (which considers optimization over trace-norm balls). We restate and prove a slightly more general version of the lemma.

\begin{lemma}
Let $f:\mbS^n\rightarrow\reals$ be $\beta$-smooth and convex. Let $\X^*\in\mS_n$ be an optimal solution of rank $r$ to the optimization problem $\min_{\X\in\mS_n}f(\X)$. Let $\lambda_1,\dots,\lambda_n$ denote the eigenvalues of $\nabla{}f(\X^*)$ in non-increasing order. Let $\zeta$ be a non-negative scalar. It holds that 
\begin{eqnarray*}
\rank(\Pi_{(1+\zeta)\mS_n}[\X^*-\beta^{-1}\nabla{}f(\X^*)]) > r \quad \Longleftrightarrow \quad \zeta >  r\beta{}(\lambda_{n-r}-\lambda_n),
\end{eqnarray*}
where $(1+\zeta)\mS_n = \{(1+\zeta)\X~|~\X\in\mS_n\}$, and $\Pi_{(1+\zeta)\mS_n}[\cdot]$ denotes the Euclidean projection onto the convex set $(1+\zeta)\mS_n$.

\end{lemma}
\begin{proof}
Let us write the eigen-decomposition of $\X^*$ as $\X^*=\sum_{i=1}^r\lambda_i^*\v_i\v_i^{\top}$. It follows from the optimality of $\X^*$ that for all $i\in[r]$, $\v_i$ is also an eigenvector of $\nabla{}f(\X^*)$ which corresponds to the smallest eigenvalue $\lambda_n$ (see Lemma 7 in \cite{Garber19}). Thus, if we let $\rho_1,\dots,\rho_n$ denote the eigenvalues (in non-increasing order) of $\Y := \X^*-\beta^{-1}\nabla{}f(\X^*)$, it holds that
\begin{eqnarray*}
\forall i\in[r]: \quad \rho_i &=& \lambda_i^* - \beta^{-1}\lambda_n;\\
\forall i>r: \quad \rho_i &=& \lambda_i^* - \beta^{-1}\lambda_{n-i+1}.
\end{eqnarray*}

Recall that $\sum_{i=1}^r\lambda_i^* =1$ and $\lambda_{r+1}^* = 0$.

It is well known that for any matrix $\M\in\mbS^n$ with eigen-decomposition $\M=\sum_{i=1}^n\sigma_i\u_i\u_i^{\top}$, the projection of $\M$ onto the set $(1+\zeta)\mS_n$, for any $\zeta \geq 0$ is given by
\begin{eqnarray*}
\Pi_{(1+\zeta)\mS_n}[\M] = \sum_{i=1}^n\max\{0,~\sigma_i-\sigma\}\u_i\u_i^{\top},
\end{eqnarray*}
where $\sigma\in\reals$ is the unique scalar such that $\sum_{i=1}^n\max\{0,~\sigma_i-\sigma\} = 1+\zeta$.

Now, we can see that $\rank(\Pi_{(1+\zeta)\mS_n}[\Y]) \leq r$ if and only if $\sigma \geq \rho_{r+1} = -\beta^{-1}\lambda_{n-r}$. 
Thus, if $\rank(\Pi_{(1+\zeta)\mS_n}[\Y]) \leq r$ then it must hold that $\sigma \geq -\beta^{-1}\lambda_{n-r}$ which implies that
\begin{align}\label{eq:robustProof:1}
1 + \zeta &= \sum_{i=1}^n\max\{0,~\rho_i-\sigma\}  = \sum_{i=1}^r\max\{0,~\rho_i-\sigma\} \leq \sum_{i=1}^r\max\{0,~\rho_i-(-\beta^{-1}\lambda_{n-r})\} \nonumber \\
&= \sum_{i=1}^r(\rho_i-(-\beta^{-1}\lambda_{n-r})) = \sum_{i=1}^r(\lambda_i +\beta(\lambda_{n-r}-\lambda_n)) = 1 + \beta{}r(\lambda_{n-r}-\lambda_{n}).
\end{align}
However, \eqref{eq:robustProof:1} can hold only if $\zeta \leq \beta{}r(\lambda_{n-r}-\lambda_n)$. Thus, we have $\rank(\Pi_{(1+\zeta)\mS_n}[\Y]) \leq r \Longrightarrow \zeta \leq \beta{}r(\lambda_{n-r}-\lambda_n)$.

On the other-hand,   if $\rank(\Pi_{(1+\zeta)\mS_n}[\Y]) > r$ then it must hold that $\sigma < -\beta^{-1}\lambda_{n-r}$ which, using the same arguments as above, implies that
\begin{align}\label{eq:robustProof:2}
1 + \zeta &= \sum_{i=1}^n\max\{0,~\rho_i-\sigma\}  > \sum_{i=1}^r\max\{0,~\rho_i-(-\beta^{-1}\lambda_{n-r})\} =1 + \beta{}r(\lambda_{n-r}-\lambda_{n}).
\end{align}
We see that \eqref{eq:robustProof:2} can hold only if $\zeta  > \beta{}r(\lambda_{n-r}-\lambda_n)$. Thus, we also have $\rank(\Pi_{(1+\zeta)\mS_n}[\Y]) > r \Longrightarrow \zeta > \beta{}r(\lambda_{n-r}-\lambda_n)$, and the lemma follows.
\end{proof}

\section{Proof of Lemma \ref{lem:robustGapAss}}

We first restate the lemma and then prove it.
\begin{lemma}
Let $f:\mbS^n\rightarrow\reals$ be $\beta$-smooth and convex. Suppose that Assumption \ref{ass:gap} holds w.r.t. $f(\cdot)$ with some parameter $\delta >0$. Let $\tilde{f}:\mbS^n\rightarrow\reals$ be differentiable and convex, and suppose that $\sup_{\X\in\mS_n}\Vert{\nabla{}f(\X) - \nabla\tilde{f}(\X)}\Vert_F \leq \nu$, for some $\nu > 0$. Then, for $\nu <  \frac{1}{2}(1+\frac{2\beta}{\delta})^{-1}\delta$, Assumption \ref{ass:gap} holds w.r.t. the function $\tilde{f}(\cdot)$ with parameter $\tilde{\delta} = \delta -  2\nu(1+\frac{2\beta}{\delta}) > 0$.
\end{lemma}
\begin{proof}
Let $\X^*$ and $\tilde{\X}^*$ denote minimizers of $f(\cdot)$ and $\tilde{f}(\cdot)$ over $\mS_n$, respectively. 
Since Assumption \ref{ass:gap} holds w.r.t. $f(\cdot)$, using the quadratic growth result of Lemma \ref{lem:genQG} we have that
\begin{align*}
\Vert{\tilde{\X}^* - \X^*}\Vert_F^2 &\leq \frac{2}{\delta}\left({f(\tilde{\X}^*)-f(\X^*)}\right) \underset{(a)}{\leq} \frac{2}{\delta}\langle{\tilde{\X}^*-\X^*, \nabla{}f(\tilde{\X}^*)}\rangle \\
&= \frac{2}{\delta}\left({\langle{\tilde{\X}^*-\X^*, \nabla{}\tilde{f}(\tilde{\X}^*)}\rangle + \langle{\tilde{\X}^*-\X^*, \nabla{}f(\tilde{\X}^*)-\nabla{}\tilde{f}(\tilde{\X}^*)}\rangle}\right)\\
&\underset{(b)}{\leq} \frac{2}{\delta} \langle{\tilde{\X}^*-\X^*, \nabla{}f(\tilde{\X}^*) - \nabla{}\tilde{f}(\tilde{\X}^*)\rangle}
\underset{(c)}{\leq} \frac{2\nu}{\delta}\Vert{\tilde{\X}^*-\X^*}\Vert_F,
\end{align*}
where $(a)$ follows from convexity of $f(\cdot)$, $(b)$ follows from optimality of $\tilde{\X}^*$ w.r.t. $\tilde{f}(\cdot)$, and $(c)$ follows from the Cauchy-Schwarz inequality and the assumption of the lemma that $\sup_{\X\in\mS_n}\Vert{\nabla{}f(\X) - \nabla\tilde{f}(\X)}\Vert_F \leq \nu$.

Thus, we get that $\Vert{\tilde{\X}^*-\X^*}\Vert_F \leq \frac{2\nu}{\delta}$.

Using Weyl's inequality for the eigenvalues we have that
\begin{align}\label{lem:perturbGap:1}
\lambda_{n}(\nabla{}\tilde{f}(\tilde{\X}^*)) &\leq \lambda_n(\nabla{}f(\X^*)) + \Vert{\nabla{}\tilde{f}(\tilde{\X}^*) - \nabla{}f(\X^*)}\Vert_F  \nonumber\\
&= \lambda_{n-1}(\nabla{}f(\X^*)) - \delta +  \Vert{\nabla{}\tilde{f}(\tilde{\X}^*) - \nabla{}f(\X^*)}\Vert_F \nonumber \\
&\leq \lambda_{n-1}(\nabla{}\tilde{f}(\tilde{\X}^*)) - \delta +  2\Vert{\nabla{}\tilde{f}(\tilde{\X}^*) - \nabla{}f(\X^*)}\Vert_F.
\end{align}
Using the smoothness of $f(\cdot)$ and the assumption $\sup_{\X\in\mS_n}\Vert{\nabla{}f(\X) - \nabla\tilde{f}(\X)}\Vert_F \leq \nu$, we have that
\begin{align}\label{lem:perturbGap:2}
\Vert{\nabla{}\tilde{f}(\tilde{\X}^*) - \nabla{}f(\X^*)}\Vert_F &\leq \Vert{\nabla{}\tilde{f}(\tilde{\X}^*) - \nabla{}f(\tilde{\X}^*)}\Vert_F +
\Vert{\nabla{}f(\tilde{\X}^*) - \nabla{}f(\X^*)}\Vert_F \nonumber \\
&\leq \nu + \beta\Vert{\tilde{\X}^*-\X^*}\Vert_F \leq \nu\left({1+\frac{2\beta}{\delta}}\right).
\end{align}
Plugging-in \eqref{lem:perturbGap:2} into \eqref{lem:perturbGap:1} and rearranging we obtain
\begin{align*}
\lambda_{n-1}(\nabla{}\tilde{f}(\tilde{\X}^*))  - \lambda_{n}(\nabla{}\tilde{f}(\tilde{\X}^*)) \geq \delta - 2\nu\left({1+\frac{2\beta}{\delta}}\right).
\end{align*}
Thus, Assumption \ref{ass:gap} indeed holds w.r.t. $\tilde{f}(\cdot)$ whenever $\nu < \frac{1}{2}(1+\frac{2\beta}{\delta})^{-1}\delta$.
\end{proof}

\bibliographystyle{plain}
\bibliography{bib}   

%
%
\end{document}

%% file: defs.tex
\newtheorem{theorem} {Theorem}
\newtheorem{lemma} {Lemma}
\newtheorem{definition} {Definition}

\newtheorem{assumption} {Assumption}

\def\n{{\mathbf{n}}}
\def\c{{\mathbf{c}}}
\def\x{{\mathbf{x}}}

\def\a{{\mathbf{a}}}
\def\u{{\mathbf{u}}}
\def\v{{\mathbf{v}}}
\def\z{{\mathbf{z}}}

\def\y{{\mathbf{y}}}

\def\b{{\mathbf{b}}}
\def\X{{\mathbf{X}}}
\def\Y{{\mathbf{Y}}}
\def\A{{\mathbf{A}}}

\def\M{{\mathbf{M}}}

\def\G{{\mathbf{G}}}
\def\C{{\mathbf{C}}}

\def\Z{{\mathbf{Z}}}

\def\EV{{\mathbf{EV}}}

\newcommand{\mW}{\mathcal{W}}

\newcommand{\mA}{\mathcal{A}}

\newcommand{\mK}{\mathcal{K}}
\newcommand{\mS}{\mathcal{S}}

\newcommand{\mbS}{\mathbb{S}}

\newcommand{\trace}{\textrm{Tr}}
\newcommand{\rank}{\textrm{rank}}
\newcommand{\reals}{\mathbb{R}}